\documentclass{amsart}
\usepackage{graphics}
\usepackage{ amssymb }
\usepackage{graphicx}
\usepackage{amsfonts,amsmath,mathrsfs}
\begin{document}

 \newtheorem{thm}{Theorem}[section]
 \newtheorem{cor}[thm]{Corollary}
 \newtheorem{lem}[thm]{Lemma}{\rm}
 \newtheorem{prop}[thm]{Proposition}

 \newtheorem{defn}[thm]{Definition}{\rm}
 \newtheorem{assumption}[thm]{Assumption}
 \newtheorem{rem}[thm]{Remark}
 \newtheorem{ex}{Example}
\numberwithin{equation}{section}
\def\la{\langle}
\def\ra{\rangle}
\def\glexe{\leq_{gl}\,}
\def\glex{<_{gl}\,}
\def\e{{\rm e}}

\def\fac{{\rm !}}
\def\x{\mathbf{x}}
\def\a{\mathbf{a}}
\def\P{\mathbb{P}}
\def\S{\mathbf{S}}
\def\h{\mathbf{h}}
\def\y{\mathbf{y}}
\def\bz{\mathbf{z}}
\def\F{\mathcal{F}}
\def\R{\mathbb{R}}
\def\T{\mathbf{T}}
\def\N{\mathbb{N}}
\def\D{\mathbf{D}}
\def\V{\mathbf{V}}
\def\U{\mathbf{U}}
\def\K{\mathbf{K}}
\def\Q{\mathbf{Q}}
\def\M{\mathbf{M}}
\def\oM{\overline{\mathbf{M}}}
\def\O{\mathbf{O}}
\def\C{\mathbf{C}}
\def\P{\mathbb{P}}
\def\Z{\mathbb{Z}}
\def\H{\mathcal{H}}
\def\A{\mathbf{A}}
\def\V{\mathbf{V}}
\def\AA{\overline{\mathbf{A}}}
\def\B{\mathbf{B}}
\def\c{\mathbf{c}}
\def\L{\mathcal{L}}
\def\bS{\mathbf{S}}
\def\H{\mathcal{H}}
\def\I{\mathbf{I}}
\def\Y{\mathbf{Y}}
\def\X{\mathbf{X}}
\def\G{\mathbf{G}}
\def\f{\mathbf{f}}
\def\z{\mathbf{z}}
\def\v{\mathbf{v}}
\def\m{\mathbf{m}}
\def\y{\mathbf{y}}
\def\d{\mathbf{d}}
\def\x{\mathbf{x}}
\def\bI{\mathbf{I}}
\def\y{\mathbf{y}}
\def\g{\mathbf{g}}
\def\w{\mathbf{w}}
\def\b{\mathbf{b}}
\def\a{\mathbf{a}}
\def\u{\mathbf{u}}
\def\q{\mathbf{q}}
\def\e{\mathbf{e}}
\def\s{\mathcal{S}}
\def\cc{\mathcal{C}}
\def\co{{\rm co}\,}
\def\tg{\tilde{g}}
\def\tx{\tilde{\x}}
\def\tg{\tilde{g}}
\def\tA{\tilde{\A}}
\def\bell{\boldsymbol{\ell}}
\def\bxi{\boldsymbol{\xi}}
\def\bpsi{\boldsymbol{\psi}}
\def\supmu{{\rm supp}\,\mu}
\def\supp{{\rm supp}\,}
\def\cd{\mathcal{C}_d}
\def\cok{\mathcal{C}_{\K}}
\def\cop{COP}
\def\vol{{\rm vol}\,}
\def\om{\mathbf{\Omega}}
\def\f{\mathscr{F}}
\def\ms{\mathcal{S}}
\def\det{{\rm det}}
\def\nab1{\nabla f^{-1}}

\title[]{Non-negative forms, 
volumes of sublevel sets, 
complete monotonicity and  moment matrices}

\author{Khazhgali Kozhasov and Jean B. Lasserre}

\address{Technische Universit\"at Braunschweig, Institut f\"ur Analysis und Algebra, Universit\"atsplatz 2, 38106 Braunschweig, Germany\medskip}

\email{k.kozhasov@tu-braunschweig.de}

\address{LAAS-CNRS and Institute of Mathematics\\
University of Toulouse\\
LAAS, 7 avenue du Colonel Roche\\
31077 Toulouse C\'edex 4, France\medskip}

\email{lasserre@laas.fr}

\thanks{J.B. Lasserre is supported by the AI Interdisciplinary Institute ANITI  funding through the french program ``Investing for the Future PI3A" under the grant agreement number ANR-19-PI3A-0004. He is also affiliated with  IPAL-CNRS laboratory, Singapore.}

\date{}
\begin{abstract}
Let $\mathcal{C}_{d,n}$ be the convex cone consisting of real $n$-variate degree $d$ forms that are strictly positive on $\mathbb{R}^n\setminus \{\mathbf{0}\}$. We prove that the Lebesgue volume of the sublevel set $\{g\leq 1\}$ of $g\in \mathcal{C}_{d,n}$ is a completely monotone function on $\mathcal{C}_{d,n}$ and investigate the related properties. Furthermore, we provide (partial) characterization of forms, whose sublevel sets have finite Lebesgue volume. Finally, we discover an interesting property of a centered Gaussian distribution, establishing a connection between the matrix of its degree $d$ moments and the quadratic form given by the inverse of its covariance matrix.

\end{abstract}

\maketitle

\vspace{-0.57cm}

\section{Introduction and main results}

We bring together various constructions and studies associated with real homogeneous polynomials, see \cite{AKU, KL2020, KMS, nonGaussian, parsimony, MSUZ, Mor, Nesterov2000, ScottSokal}. 
Our discoveries in Subsection \ref{sub:CM} enrich the existing connection \cite{KMS, MSUZ, ScottSokal} between \emph{real algebraic geometry} and \emph{the theory of completely monotone functions}. Theorem \ref{thm:moments} establishes a property of a \emph{centered Gaussian distribution} that we also interpret in the context of \emph{polynomial optimization} \cite{Nesterov2000}. In Subsection \ref{sub:L2L1} we solve a problem that contributes to a general study of \emph{extremal properties of homogeneous polynomials} (see \cite{AKU, KL2020, KT2022, parsimony}). Finally, the results stated in Subsection \ref{sub:volume} naturally complement investigations from \cite{KL2020, parsimony} about \emph{volumes of sublevel sets of non-negative homogeneous polynomials}.   
\smallskip 

Let $\H_{d,n}$ denote the space of real $n$-variate forms (homogeneous polynomials) of degree $d$. The space $\H_{d,n}$ is endowed with \emph{the Bombieri inner product} 
\begin{align}\label{eq:Bomb}
    \langle g, h\rangle\ =\ \sum_{\vert\alpha\vert=d} {d\choose \alpha}^{-1} g_{\alpha} h_{\alpha},\quad {d\choose\alpha}=\frac{d\mathrm{!}}{\alpha_1\mathrm{!}\cdots\alpha_n\mathrm{!}},
\end{align}
where  $g(\x)=\sum_{\vert \mathbf{\alpha}\vert=d} g_{\mathbf{\alpha}} \x^{\mathbf\alpha}$ and $h(\x)=\sum_{\vert \mathbf{\alpha}\vert=d} h_{\mathbf{\alpha}} \x^{\mathbf\alpha}$ are two forms written in the basis of monomials of degree $d$. 
For a form $g\in \H_{d,n}$ we consider its sublevel set \[\G=\{\x\in \R^n:g(\x)\leq 1\}.\]
We are  interested in \emph{the volume function} $f$ that to a given $g\in \H_{d,n}$ associates the Lebesgue volume $f(g)=\vol(\G)$ of its sublevel set.
The volume of $\G$ is infinite if $g$ takes negative values. 
Thus, it is natural to apply $f$ only to polynomials that are non-negative on $\R^n$. 
We call a form $g$ \emph{positive definite} (PD for short), if it is positive on $\R^n\setminus\{\boldsymbol{0}\}$. The sublevel set $\G\subset \R^n$ of a PD form is compact and hence it has a finite Lebesgue volume.
In this regard, we are concerned with the open convex cone $\mathcal{C}_{d,n}\subset \H_{d,n}$  of PD forms. 
We implicitly assume that the degree $d$ is even, as only in this case $\mathcal{C}_{d,n}$ is non-empty.
The closure $\overline{\mathcal{C}_{d,n}}$ of $\mathcal{C}_{d,n}$ with respect to the norm topology on $(\H_{d,n},\langle\cdot,\cdot\rangle)$ consists of forms that are non-negative on $\mathbb{R}^n$.

\subsection{Complete monotonicity of the volume function}\label{sub:CM}

The cone $\mathcal{C}_{d,n}$ is a natural domain of definition of the volume function
\begin{equation}\begin{aligned}\label{eq:volume_function}
    f: \mathcal{C}_{d,n}\ &\rightarrow\ \R,\\
    g\ &\mapsto\ \vol(\G)\ =\ \int_{\G}\textrm{d}\x,
\end{aligned}\end{equation}
which is strictly convex and admits the following integral representation
\begin{align}\label{eq:int}
   f(g)\ =\ \vol(\G)\ =\ \frac{1}{\Gamma\left(1+n/d\right)}\int_{\R^n} \exp(-g(\x))\,\textrm{d}\x,
\end{align}
see \cite[Thm. 2.2]{parsimony}.
In our first result we show that \eqref{eq:volume_function} is \emph{completely monotone},
that is, it is \emph{the Laplace transform} of some Borel measure on the closed dual cone $\mathcal{C}_{d,n}^{\,*}=\{L\in \H_{d,n}^*: L(g)\geq 0\ \forall g\in \mathcal{C}_{d,n}\}$ to $\mathcal{C}_{d,n}$.
To state it, let us recall (from, e.g., \cite[Thm. 19]{DIDIO2022126066}) that  $\mathcal{C}_{d,n}^{\,*}$ is the conic hull of the image of the \emph{Veronese map}
\begin{equation}\label{eq:Veronese}
\begin{aligned}
     \boldsymbol\Theta_{d,n}:\R^n\ &\rightarrow\ \mathcal{C}_{d,n}^{\,*}\subset \H_{d,n}^*,\\
     \boldsymbol{\ell}\ &\mapsto\ [g\mapsto g(\boldsymbol{\ell})].
\end{aligned}
\end{equation}
By \cite[Thm. 17.10]{schmudgen2017moment}, $\mathcal{C}^{\,*}_{d,n}$, also known as \emph{the moment cone}, consists of \emph{(truncated) moment functionals}, that is, such $L\in \H_{d,n}^*$ with $L(g)=\int_{\mathbb{S}^{n-1}} g(\z)\,\textrm{d}\nu(\z)$, $g\in \H_{d,n}$, for some measure $\nu$ supported on \emph{the unit sphere} $\mathbb{S}^{n-1}=\{\z\in \R^n\,:\, \z^{\mathsf T}\z=1\}$.
Then, an element $\boldsymbol\Theta_{d,n}(\ell)$, $\ell\in \mathbb{S}^{n-1}$, corresponds to \emph{the Dirac measure} at $\ell$.
An alternative perspective on $\mathcal{C}_{d,n}^{\,*}$ comes with the identification of $\H_{d,n}^*$ and $\H_{d,n}$ via the inner product \eqref{eq:Bomb}. 
Under this identification, one has $g(\boldsymbol{\ell})=\big\langle \big(\boldsymbol{\ell}^{\mathsf T} \cdot\big)^d, g\big\rangle$ (see, e.g., \cite[(19.6)]{schmudgen2017moment}) and, in particular, \eqref{eq:Veronese} sends a vector $\boldsymbol{\ell}\in \R^n$ to the $d$th power of a linear form, $\theta_{\boldsymbol\ell}=\big(\boldsymbol{\ell}^{\mathsf T}\cdot\big)^d\in \mathcal{C}_{d,n}^{\,*}\subset \H_{d,n}^*\simeq \H_{d,n}$.
\begin{thm}\label{thm:CM}
For $n\geq 2$ and any even $d\geq 2$ the volume function \eqref{eq:volume_function} admits an integral representation
\begin{align}\label{eq:int_rep}
f(g)\ =\ \int_{\mathcal{C}_{d,n}^{\,*}} \exp(-\langle \theta, g\rangle)\, \frac{\textrm{\normalfont{d}}\mu(\theta)}{\Gamma\left(1+n/d\right)},
\end{align}
where $\mu$ is the push-forward measure of the Lebesgue measure on $\R^n$ under the Veronese map \eqref{eq:Veronese}.
In particular, $f: \mathcal{C}_{d,n}\rightarrow \R$ is completely monotone and for all $k\in\N$ and $v_1,\dots, v_k\in \mathcal{C}_{d,n}$ we have that
\begin{align}\label{eq:der}
    (-1)^k D_{v_1}\dots D_{v_k} f(g)\ \geq\ 0,\quad g\in \mathcal{C}_{d,n},
\end{align}
where $D_v$ denotes the directional derivative along $v\in \mathcal{C}_{d,n}$.
\end{thm}

\begin{rem}
  Complete monotonicity of a differentiable function $f: C\rightarrow \R$ on an open cone $C\subset \R^N$ is normally defined in terms of conditions on $k$-fold directional derivatives of $f$ as in \eqref{eq:der}. 
Thanks to Bernstein-Hausdorff-Widder-Choquet  theorem, recalled in Section \ref{sec:CM}, this is equivalent to the possibility of representing $f$ as the Laplace transform of some Borel measure on the dual cone $C^*$. 
\end{rem}

Scott and Sokal \cite{ScottSokal} investigated complete monotonicity of negative powers of some combinatorially defined polynomials .
In \cite[Theorem 1.3]{ScottSokal} they characterized values of $s$ for which the function $G\mapsto (\det (G))^{-s}$, defined on the cone of positive definite matrices, is completely monotone.  
The volume function $f(g)$ associated to a quadratic form $g(\x)=\x^{\mathsf T} G \x$ is proportional to $(\det(G))^{-1/2}$, see Proposition \ref{prop:2d=2}. 
Thus, in the case $d=2$ complete monotonicity of the volume function \eqref{eq:volume_function} also follows from the mentioned characterization from \cite{ScottSokal}.
The results from \cite{ScottSokal} along with exponential families in algebraic statistics \cite{expo} motivated Micha\l{}ek, Sturmfels and the first author of the present work to investigate \cite{KMS} complete monotonicity of negative powers $g^{-s}$ of \emph{hyperbolic polynomials}. 
A particular focus in their study was given to  elementary symmetric polynomials and to products of linear forms. 

In physics one often considers exponential probability density functions of the form  $\frac{1}{Z_g}\exp(-g(\x))$, where $Z_g=\int_{\R^n} \exp(-g(\x))\,\textrm{d}\x$ is the normalization constant, also known as \emph{the partition function}.
It is in general difficult to find closed-form expressions for $Z_g$ (respectively, for \eqref{eq:int}).
As we see in Proposition \ref{prop:2d=2}, it is possible to write $f(g)$ in terms of the discriminant of a quadratic form $g\in \mathcal{C}_{2,n}$, which is also equal to the determinant of 
the real symmetric matrix associated to $g$.
In \cite{Mor} Morozov and Shakirov develop this line further,
calling integrals in \eqref{eq:int} \emph{integral discriminants} and writing down their expressions in terms of $\textrm{SL}(n,\R)$-invariants of $g\in \H_{d,n}$.
Indeed, the standard action of \emph{the special linear group} on the space $\H_{d,n}$ preserves the integral \eqref{eq:int} and as a consequence the integral discriminant depends on $g$ only through the invariants of the action. 
For example, for $n=2$ and $d=4$ an expression for \eqref{eq:int} in terms of the $\textrm{SL}(2,\R)$-invariants is given in \cite[Section 5.1]{Mor}.

One can also generalize the above setting as explained in  the following remark.
\begin{rem}
For a non-negative form $h\in \overline{\mathcal{C}_{e,n}}$ of degree $e\geq 0$ consider a function
\begin{equation}\label{eq:CMgen}
\begin{aligned}
    f_h: \mathcal{C}_{d,n}&\rightarrow \R,\\
    g &\mapsto \int_{\G} h(\x)\,\textrm{d}\x,
\end{aligned}
\end{equation}
from which we retrieve the volume function \eqref{eq:volume_function} by setting $h=1$.
In Section \ref{sec:CM} we prove a generalization of Theorem \ref{thm:CM}, showing that \eqref{eq:CMgen} is completely monotone. 
\end{rem}
Along with proving complete monotonicity of the volume function \eqref{eq:volume_function} (and its generalization defined in the above remark), in Section \ref{sec:CM} we also discuss related properties as well as prove a version of Theorem \ref{thm:CM} for \emph{sum of squares forms}.

\subsection{Moment matrices of central Gaussian vectors and sums of squares}\label{sub:moments}

A form $g\in \H_{d,n}$ is called \emph{a sum of squares} (or \textrm{SOS}), if $g(\x)=\sum_{j=1}^r h_j(\x)$ for some forms $h_1,\dots, h_r\in \H_{d/2,n}$ of degree $d/2$.
 Equivalently, one can write $g$ as
\begin{align}\label{eq:SOS}
  g(\x)\ =\ \m_{d/2}(\x)^{\mathsf T} \,G\, \m_{d/2}(\x),\quad \x\in \R^n, 
\end{align}
where $\m_{d/2}(\x)=(\x^{\boldsymbol\alpha})_{\vert \boldsymbol\alpha\vert =d/2}$ is the column-vector of monomials\footnote{From now on we fix an order on the set of monomials of a given degree.} of degree $d/2$ and $G\in \textrm{PD}_M$ is a positive semi-definite matrix of size $M\times M$, $M={d/2+n-1 \choose n-1}$ (called \emph{a Gram matrix} of $g$). 
Below by $\textrm{PD}_M$ and $\textrm{PSD}_M$ we denote the convex cone of positive definite, respectively positive semi-definite, $M\times M$ real symmetric matrices.
By definition, sums of squares are non-negative, and, if $g$ is given by \eqref{eq:SOS} with a positive definite matrix $G$, then the form $g\in \mathcal{C}_{d,n}$ is also positive definite.
\begin{rem}
  It is well-known that not every non-negative form is a sum of squares, see \cite{Rez1996}. However, by a celebrated result of Artin \cite{Artin}, after multiplying a non-negative form by a suitable sum of squares form, one obtains a sum of squares.
  Moreover, by a result \cite[Thm. 3.12]{Reznick1995UniformDI} of Reznick, if a form $g\in \mathcal{C}_{d,n}$ is positive definite, there exists $s\in \mathbb{N}$ so that $\Vert \x\Vert^{2s} g\in\mathcal{C}_{d+2s,n}$ is a sum of squares, where $\Vert \x\Vert^2=\x^{\mathsf T}\x=\sum_{i=1}^n x_i^2$ is the Euclidean norm of $\x\in \R^n$.
  In fact, by the same result, $\Vert \x\Vert^{2s}g$ is even a sum of $(d+2s)$-th powers of linear forms and hence $\Vert \x\Vert^{2s} g\in \boldsymbol{\Theta}_{d+2s, n}(\R^n)$ is identified with a point of the moment cone $\mathcal{C}_{d+2s,n}^{\,*}$.
  Thus, sums of squares constitute an important subclass of the class of non-negative forms.
  Furthermore, with the representation \eqref{eq:SOS} for sums of squares, one can approach \emph{polynomial optimization problems} via \emph{semi-definite programming}, see \cite{lasserre_2015}.
\end{rem}

A centered multivariate normal distribution is defined by its covariance matrix $Q^{-1}$, with the probability density function given by 
\begin{align}\label{eq:Sigma} 
p^{\,}_{Q}(\y)\ =\ \frac{\sqrt{\det(Q)}}{\sqrt{(2\pi)^n}}\exp\left(-\frac{\y^{\mathsf T}Q \y}{2}\right),\quad \y\in \R^n.
\end{align}
The covariance matrix equals \emph{the Hankel matrix} of moments of $p^{\,}_Q$ of degree $2$, i.e.,
\begin{align}\label{eq:d=2case}
    Q^{-1}\ =\ \M_2[Q]\ =\ \int_{\R^n} \y\y^{\mathsf T} p^{\,}_{Q}(\y)\,\textrm{d}\y\ =\ \left( \int_{\R^n} y_iy_j p^{\,}_{Q}(\y)\,\textrm{d}\y\right).
\end{align}
Motivated by this property of Gaussian distributions, in \cite[$2.5$]{nonGaussian} the second author of the present work considers \emph{the Gaussian-like density} 
\begin{align}\label{eq:P_G}
P_G(\y)\ =\ \frac{\exp\left(-k\,\m_{d/2}(\y)^{\mathsf T} G\m_{d/2}(\y)\right)}{\int_{\R^n}\exp\left(-k\,\m_{d/2}(\z)^{\mathsf T} G\m_{d/2}(\z)\right)\textrm{d}\z},\quad \y\in \R^n,
\end{align}
associated to a positive definite \textrm{SOS} form $g(\x)=\m_{d/2}(\x)G\m_{d/2}(\x)$, $G\in \textrm{PD}_M$, and  \emph{the Hankel-like matrix}
\[
\mathcal{M}_d[G]\ =\ \int_{\R^n}\m_{d/2}(\y)\m_{d/2}(\y)^{\mathsf T} P_G(\y)\,\textrm{d}\y
\]
of degree $d$ moments of \eqref{eq:P_G}, where $(2k)^{-1}={d/2+n-1\choose n}$. When $d=2$, one has $k=1/2$ and \eqref{eq:P_G} is the Gaussian density \eqref{eq:Sigma}, whose matrix $\mathcal{M}_2[G]=\M_2[Q]$ of degree $2$ moments is the inverse of $G=Q\in\textrm{PD}_n$, see \eqref{eq:d=2case}. When $d>2$ is higher, Lasserre shows in \cite[Lemma 4]{nonGaussian} that $\mathcal{M}_d[G]=G^{-1}$ (for $G\in \textrm{PD}_M$) if and only if $G$ is a critical point of the function
\begin{align}\label{eq:function}
G\in \textrm{PD}_M\ \mapsto\ (\det (G))^k\int_{\R^n} \exp\left(-k\m_{d/2}(\x)^{\mathsf T} G \m_{d/2}(\x)\right)\textrm{d}\x.
\end{align}
Theorem \ref{thm:moments} below implies that forms $g(\x)=(\x^{\mathsf T} Q\x)^{d/2}$, $Q\in \textrm{PD}_n$, fulfil this condition.
Before we state it, for $Q\in \textrm{PD}_n$ and $d\geq 2$ consider the density
\[
p^{(d)}_{\,Q}(\y)\, =\, \frac{\sqrt{\det(Q)}}{\sigma_d^n\sqrt{(2\pi)^n}}\exp\left(-\frac{\y^{\mathsf T}Q \y}{2\sigma_d^2}\right),\quad \sigma_d\, =\, \left(\frac{{d/2+n-1\choose n-1}}{\sqrt{2}^d\prod_{i=0}^{d/2-1} \big(\frac{n}{2}+i\big)}\right)^{1/d}\hspace{-0.2cm},
\]
of a centered Gaussian distribution with covariance matrix $\sigma_d^2Q^{-1}$, and let
\begin{align}\label{eq:M_d_def}
\M_d[Q]\ =\ \int_{\R^n} \m_{d/2}(\y)\m_{d/2}(\y)^{\mathsf T} p^{(d)}_{\,Q}(\y)\,\textrm{\normalfont{d}}\y\ =\ \left(\int_{\R^n} \y^{\alpha+\beta} p^{(d)}_{\,Q}(\y)\,\textrm{\normalfont{d}}\y\right)
\end{align}
be the Hankel-like matrix of moments of $p^{(d)}_{\,Q}$ of degree $d$.
For $d=2$ one has $\sigma_2=1$ and hence we recover $p^{\,}_{Q}=p^{(2)}_{\,Q}$ given by \eqref{eq:Sigma}. 

\begin{thm}\label{thm:moments}
Let $Q\in \textrm{\normalfont{PD}}_n$ be any positive definite matrix and let $d\geq 2$ be even. Then $\M_d[Q]^{-1}$ is a Gram matrix of the \textrm{\normalfont{SOS}} form $(\x^{\mathsf T} Q\x)^{d/2}$, that is,
  \begin{align}\label{eq:claim}
    \m_{d/2}(\x)^{\mathsf T}\M_d[Q]^{-1}\m_{d/2}(\x)\ =\ \left(\x^{\mathsf T} Q\x\right)^{d/2}.
    \end{align}
    Furthermore, we have $\M_d[Q]=\mathcal{M}_d[G]$, where $\mathcal{M}_d[G]$ is the matrix of degree $d$ moments of the Gaussian-like density \eqref{eq:P_G} associated to $g(\x)=(\x^{\mathsf T} Q\x)^{d/2}$.
\end{thm}

\begin{rem}
By \cite[Lemma 4]{nonGaussian}, for any $Q\in \textrm{\normalfont{PD}}_n$ the matrix $\mathcal{M}_d[Q]^{-1}\in\textrm{\normalfont{PD}}_M$ is a critical point of the function \eqref{eq:function}. It is an interesting open question whether there exist other \textrm{\normalfont{SOS}} forms with this property.
\end{rem}

Sums of squares of degree $d$ form a closed convex subcone (called \emph{the \textrm{\normalfont{SOS}} cone}) $\Sigma_{d,n}\subseteq \overline{\mathcal{C}_{d,n}}$ of the cone of non-negative homogeneous polynomials. Its  dual cone $\Sigma_{d,n}^*=\{L\in \H^*_{d,n}\,:\,L(g)\geq 0\ \forall\, g\in \Sigma_{d,n}\}$ contains  $\mathcal{C}_{d,n}^{\,*}$ and is known as \emph{the pseudo-moment cone}. 
For any \emph{(truncated) pseudo-moment functional} $L\in \Sigma_{d,n}^*$ let  
\begin{align}\label{eq:pseudo-moment_matrix}
    \M_d(L)\ =\ L\left(\m_{d/2}(\y) \m_{d/2}(\y)^{\mathsf T}\right)\ =\ \left(L(\y^{\alpha+\beta})\right)_{\vert \alpha\vert=\vert\beta\vert=d/2}
\end{align}
be the $M\times M$ \emph{Hankel-like matrix} of pseudo-moments of degree $d$ associated to $L$.
One has  $\M_d(L)\in \textrm{PSD}_M$, as $\h^{\mathsf T} \M_d(L) \h=L(\h^{\mathsf T}\m_{d/2}(\y)\m_{d/2}(\y)^{\mathsf T}\h)=L(h^2)\geq 0$ holds for any form $h\in \H_{d/2,n}$ whose coefficients in the basis of monomials form a column vector $\h=(h_\alpha)_{\vert\alpha\vert=d/2}$.
Conversely, any matrix $\M_d\in \textrm{PSD}_M$ that satisfies $(\M_d)_{\alpha\beta}=(\M_d)_{\alpha'\beta'}$ whenever $\alpha+\beta=\alpha'+\beta'$, $\vert \alpha\vert=\vert \beta\vert=\vert\alpha'\vert=\vert\beta'\vert=d/2$, is of the form \eqref{eq:pseudo-moment_matrix} for $L\in \H_{d,n}^*$ defined by $L(\y^{\alpha+\beta})=(\M_d)_{\alpha\beta}$, $\vert\alpha\vert=\vert\beta\vert=d/2$.
Furthermore, for a linear functional $L\in \Sigma_{d,n}^*$ that lies in the interior of the pseudo-moment cone, the associated pseudo-moment matrix $\M_d(L)\in \textrm{PD}_{M}$ is positive definite and hence invertible.
If $L\in \mathcal{C}_{d,n}^{\,*}$ is even a moment functional, that is, 
\begin{align}\label{eq:moment_func}
   L(g)\ =\ \int_{\mathbb{S}^{n-1}} g(\z)\,\textrm{d}\nu(\z),\quad g\in \H_{d,n},
\end{align}
for some measure $\nu$ supported on $\mathbb{S}^{n-1}$, then the entries of the matrix $\M_d(L)$ are moments of $\nu$ of degree $d$.

By a result of Nesterov (see \cite[Thm. $2$]{Nesterov2000}), a \textrm{SOS} form $g\in \Sigma_{d,n}$ belongs to the interior of the cone $\Sigma_{d,n}$ if and only if there is some linear functional $L\in \Sigma_{d,n}^*$ in the interior of the pseudo-moment cone such that $g(\x)=\m_{d/2}(\x)^{\mathsf T}\,\M_d(L)^{-1}\m_{d/2}(\x)$. Moreover, $L\in \Sigma_{d,n}^*$ satisfying this condition is unique.
\begin{rem}
Theorem \ref{thm:moments} delivers Nesterov's characterization for a power of a quadratic form $g(\x)=(\x^{\mathsf T} Q\x)^{d/2}$, associated to a positive definite matrix $Q\in\textrm{\normalfont{PD}}_n$.
It implies that the unique $\M_d(L)\in \textrm{\normalfont{PD}}_M$ is the matrix of degree $d$ moments of the centered Gaussian vector in $\R^n$, whose covariance matrix is proportional to $Q^{-1}$.
\end{rem}

By integrating out the radial part of the Gaussian measure from Theorem \ref{thm:moments}, one also obtains a measure on $\mathbb{S}^{n-1}$, whose degree $d$ moments comprise $\M_d(L)$. 
\begin{cor}\label{cor:Nesterov}
Let $Q\in \mathrm{\normalfont{PD}}_n$ be a positive definite matrix and let $L$ be the moment functional \eqref{eq:moment_func}, whose associated measure $\nu$ on $\mathbb{S}^{n-1}$ is defined by
\begin{align*}
    \textrm{\normalfont{d}}\nu(\z)\ =\ {d/2+n-1\choose n-1}\frac{\sqrt{\det(Q)}}{\sqrt{\z^{\mathsf T} Q \z}^{\,d+n}}\frac{\textrm{\normalfont{d}}\mathbb{S}^{n-1}(\z)}{\vol(\mathbb{S}^{n-1})},\quad \z\in \mathbb{S}^{n-1},
\end{align*}
where $\textrm{\normalfont{d}}\mathbb{S}^{n-1}$ is the standard Riemannian measure on the Euclidean sphere with $\vol(\mathbb{S}^{n-1})=2\sqrt{\pi}^n/\,\Gamma\left(n/2\right)$ being its total volume.
Then,  $\M_d(L)^{-1}$ is a Gram matrix of the \textrm{\normalfont{SOS}} form $(\x^{\mathsf T} Q\x)^{d/2}$, that is,
  \begin{align*}    \m_{d/2}(\x)^{\mathsf T}\M_d(L)^{-1}\m_{d/2}(\x)\ =\ \left(\x^{\mathsf T} Q\x\right)^{d/2}.
    \end{align*}
\end{cor}

Theorem \ref{thm:moments} together with Corollary \ref{cor:Nesterov} are proven in Section \ref{sec:moments}.

\subsection{Comparison of $L^1$- and $L^2$-norms associated to a PD form}\label{sub:L2L1}

Given a positive definite form $g\in \mathcal{C}_{d,n}$, let us denote by $\mu_g$ the Lebesgue measure restricted to the sublevel set $\G\subset \R^n$ of $g$. 
Integrating $\mu_g$ against any form $h\in \H_{e,n}$ of degree $e$, one obtains by \cite[(2.1)]{parsimony} that
\[\mu_g(h)\ =\ \int_{\G} h(\x)\,\textrm{d}\x\ =\ \frac{1}{\Gamma\left(1+(n+e)/d\right)}\int_{\R^n} h(\x) \exp(-g(\x))\,\textrm{d}\x.\]
In this perspective, $\mu_g$ can be identified with a measure with exponential density $\exp(-g(\x))$ with respect to the Lebesgue measure on $\R^n$. 
We consider $L^1$- and $L^2$-norms associated to the measure $\mu_g$,
\begin{align}
    \Vert h\Vert_{L^1(\mu_g)}\ =\ \int_{\G} \vert h(\x)\vert \,\textrm{d}\x,\quad \Vert h\Vert_{L^2(\mu_g)}\ =\ \left(\int_{\G} \vert h(\x)\vert^2\,\textrm{d}\x\right)^{1/2},
\end{align}
where $h$ is any Lebesgue measurable function on $\G$.
We show that $g$ minimizes the ratio of $L^2$- and $L^1$-norms over all non-zero $h\in \H_{d,n}$.

\begin{thm}\label{thm:L2L1}
Let $g\in \mathcal{C}_{d,n}$. Then 
\[ g\ =\ \arg\min_{h\in \H_{d,n}}\left\{\Vert h\Vert_{L^2(\mu_g)}\,:\,\Vert h\Vert_{L^1(\mu_g)} = \Vert g\Vert_{L^1(\mu_g)}\right\} \]
or, equivalently,
\[ \frac{\Vert g\Vert_{L^2(\mu_g)}}{\Vert g\Vert_{L^1(\mu_g)}}\ =\ \min\left\{\frac{\Vert h\Vert_{L^2(\mu_g)}}{\Vert h\Vert_{L^1(\mu_g)}}\,:\, h\in \H_{d,n}\setminus \{0\}\right\}.\]
Moreover, up to a factor, $g$ is the unique minimizer of these optimization problems.
\end{thm}

In particular, the positive definite form $g\in \mathcal{C}_{d,n}$ has the smallest $L^2(\mu_g)$-norm among all degree $d$ forms $h\in \H_{d,n}$  that have the same $L^1(\mu_g)$-norm as $g$.

\begin{rem}
The above result also holds when $g$ and $h$ are positively homogeneous functions of degree $d$ with $g(\x)$ being positive for $\x\neq \mathbf{0}$, see \cite[Lemma 3]{nonGaussian} and our proof of Theorem \ref{thm:L2L1} in Section \ref{sec:L2L1}. \end{rem}


\subsection{Non-negative forms with sublevel sets of finite volume}\label{sub:volume}

It turns out that, for some non-negative forms $g\in \partial \mathcal{C}_{d,n}=\overline{\mathcal{C}_{d,n}}\setminus \mathcal{C}_{d,n}$ in the boundary of $\mathcal{C}_{d,n}$, the sublevel set $\G$ has finite Lebesgue volume $f(g)=\vol(\G)$, while it is not the case for all $g\in \partial \mathcal{C}_{d,n}$ for general $d$ and $n$.
Motivated by this observation, we consider the set $\mathcal{V}_{d,n}$ of forms $g\in \H_{d,n}$ with $f(g)<\infty$.
One naturally has $\mathcal{C}_{d,n}\subseteq \mathcal{V}_{d,n}\subseteq \overline{\mathcal{C}_{d,n}}$ and $\mathcal{V}_{d,n}\subset \H_{d,n}$ is a convex cone by \cite[Thm. 2.1]{parsimony}. 
It is straightforward to see that $\mathcal{V}_{2,n}=\mathcal{C}_{2,n}$, that is, only positive definite quadratic forms $g\in \overline{\mathcal{C}_{2,n}}$ have finite $f(g)$.
We give a complete characterization of binary forms in $\mathcal{V}_{d,2}$ in terms of multiplicities of their real zeros.

\begin{thm}\label{thm:binary}
 A non-negative binary form $g\in \overline{\mathcal{C}_{d,2}}$ is in $\mathcal{V}_{d,2}$ if and only if $g$ has zeros of order at most $d/2-1$. In particular, $\mathcal{V}_{4,2}=\mathcal{C}_{4,2}$.
\end{thm}

We also provide some sufficient conditions for membership in $\mathcal{V}_{d,n}$ for $n>2$. 
For this let us call a non-negative form $g\in \overline{C}_{d,n}$ \emph{generic}, if it is \emph{round at every its real zero $\x\in \R^n\setminus \{0\}$} (see \cite[p. 47]{Iliman2014}), that is, the Hessian matrix $\mathrm{Hess}_{\x} g=\left(\frac{\partial^2 g}{\partial x_i \partial x_j}(\x)\right)$ is positive definite when restricted to the orthogonal complement of $\x$
(equivalently, $\mathrm{Hess}_{\x} g$ is of corank one). 
Thus, a non-negative (but not positive definite) quadratic form $g(\x)=\x^{\mathsf T} G \x$ is generic if and only if the associated positive semi-definite matrix $G$ is of corank one.
Also, for a non-negative binary form $g\in \overline{C_{d,2}}$ the condition of being generic means that all real zeros of $g$ are of order two.

\begin{thm}\label{thm:generic}
For $d\geq 4$, $n\geq 3$ a generic non-negative form $g\in\overline{\mathcal{C}_{d,n}}$ is in $\mathcal{V}_{d,n}$.
\end{thm}

Theorems \ref{thm:binary} and \ref{thm:generic} are proven in Section \ref{sec:volume}.
Now, we consider some examples of generic non-negative forms that lie on the boundary of $\mathcal{C}_{d,n}$. 
\begin{ex}
The \emph{Motzkin form} 
\begin{align}\label{eq:Motzkin} 
g(\x)\ =\ x_1^4x_2^2+x_1^2x_2^4+x_3^6-3x_1^2x_2^2x_3^2,\quad \x=(x_1,x_2,x_3),
\end{align} 
was historically the first explicit example of a non-negative form that is not a sum of squares, see \cite{Rez1996}. A point $\x\in \R^3\setminus\{\mathbf{0}\}$ is a zero of $g$ if and only if $\vert x_1\vert=\vert x_2\vert=\vert x_3\vert$. Setting $x_3=1$ in $g$ gives \emph{the Motzkin polynomial} $\tilde g(\y)=y_1^4y_2^2+y_1^2y_2^4+1-3y_1^2y_2^2$, whose real zeros are $(1,1)$, $(1,-1)$, $(-1,1)$, $(-1,-1)$. At any  real zero $\x\in \R^3\setminus \{\mathbf{0}\}$ (say, $\x=(1,1,1)$) of $g$ the Hessian $\textrm{\normalfont{Hess}}_{\x} g$ is positive semi-definite and has rank $2$ and hence the Motzkin form $g\in \partial\mathcal{C}_{6,3}$ is generic in the above sense.
\end{ex}

As the following example shows, it is not difficult to find generic forms in $\overline{\mathcal{C}_{d,n}}$ also for higher $n\geq 3$ and even $d\geq 4$.

\begin{ex}\label{ex:series}
Consider a sum of squares form \[g(\x)\ =\ x_n^{d-2}\sum_{i=1}^{n-1} x_i^2+\frac{2}{d}\sum_{i=1}^{n-1} x_i^d,\quad \x=(x_1,\dots, x_{n-1},x_n).\] 
Then a real zero $\x\in \R^n$ of $g$ must satisfy $x_1=\dots=x_{n-1}=0$. At $\x=(0,\dots, 0, 1)$ the Hessian of $g$ has rank $n-1$ and hence $g\in \partial\mathcal{C}_{d,n}$ is generic.
\end{ex}

By Theorem \ref{thm:generic}, the sublevel sets of the Motzkin form \eqref{eq:Motzkin} and of forms constructed in Example \ref{ex:series} are of finite Lebegue volume.

One might wonder, whether the term ``generic'' is in place in the above context. 
The following remark motivates our choice for the terminology.
\begin{rem}
The boundary $\partial\mathcal{C}_{d,n}\subset \H_{d,n}$ of the cone of non-negative forms is a semialgebraic subset of codimension one. 
In Proposition \ref{prop:nongeneric} we show that the set of non-generic forms is a semialgebraic subset of $\H_{d,n}$ of codimension at least $2$ and that it is nowhere dense in the boundary $\partial \mathcal{C}_{d,n}\subset \H_{d,n}$ of $\mathcal{C}_{d,n}$, endowed with the topology induced  from $(\H_{d,n},\langle\cdot,\cdot\rangle)$.  
\end{rem}

\smallskip
\begin{center}
    \textbf{Structure of the paper}
\end{center}
\smallskip

There are 4 sections, each corresponding to a subsection of the above introductory section. Thus, in the next section we prove complete monotonicity of the volume function \eqref{eq:volume_function} and discuss a generalization of this result as well as treat the case of sums of squares. In Section \ref{sec:moments} we discuss moment matrices of centered Gaussian vectors and prove Theorem \ref{thm:moments} and Corollary \ref{cor:Nesterov}. In Section \ref{sec:L2L1} a proof of the extremal property from Theorem \ref{thm:L2L1} is presented. And in the last section we characterize non-negative forms, whose sublevel sets have finite Lebesgue volume.  
\medskip

\noindent{\bf Acknowledgements:} we are thankful to Boulos El Hilany for helpful discussions.
\smallskip

\section{Complete monotonicity of the volume function}\label{sec:CM}

Let $C\subset \R^N$ be an open convex cone and denote by $C^*$ \emph{the dual cone} to $C$,
\[C^*\ =\ \left\{\mathbf{L}\in \left(\R^N\right)^*\,:\, \mathbf{L}(\mathbf{g})\geq 0\ \forall\ \mathbf{g}\in C\right\}.\] 
\begin{defn}
  A function $f:C\to\R$ is \emph{completely monotone}, if it is $C^\infty$-differentiable and for all $k\in\N$ and all vectors $\mathbf{v}_1,\ldots,\mathbf{v}_k\in C$
\begin{align}\label{eq:CM}
 (-1)^k D_{\mathbf{v}_1}\cdots D_{\mathbf{v}_k}\, f(\mathbf{g})\,\geq\,0,\quad \mathbf{g}\in C,
\end{align}
where $D_{\mathbf{v}}$ denotes the directional derivative along the vector $\mathbf{v}$.
\end{defn}
The Bernstein-Hausdorff-Widder theorem \cite[Thm. 12a]{Widder} gives a characterization of completely
monotone functions in one variable, these are exactly Laplace transforms of Borel
measures on the positive reals. Choquet \cite[Thm. 10]{Choquet} found a generalization of this result to  convex cones in higher dimensional spaces.
\begin{thm}[Bernstein-Hausdorff-Widder-Choquet theorem]\label{thm:BHWC}
A smooth function $f:C\rightarrow  \R$ on an open convex cone $C\subset \R^N$ is completely monotone if and only if it is \emph{the Laplace transform} of a unique Borel measure $\mu$ supported on  $C^*$, that is,
\begin{align}\label{eq:BHWC}
    f (\mathbf{g})\ =\ \int_{C^*} \exp(-\mathbf{L}(\mathbf{g}))\, \textrm{\normalfont{d}}\mu(\mathbf{g}).
    \end{align}
\end{thm}

For a non-negative form $h\in \overline{\mathcal{C}_{e,n}}$ of degree $e\geq 0$ the function $f_h$ defined in \eqref{eq:CMgen} admits an integral representation (see \cite[Thm. 1]{nonGaussian})
\begin{align}\label{eq:repgen}
    f_h(g)\ =\ \int_{\G} h(\x)\,\textrm{d}\x\ =\ \frac{1}{\Gamma(1+(n+e)/d)}\,\int_{\R^n} h(\x)\exp(-g)\,\textrm{d}\x,\quad g\in \mathcal{C}_{d,n}.
\end{align}
In particular, the Lebesgue volume of the sublevel set $\G=\{\x\in \R^n:g(\x)\leq 1\}$ of $g$ is obtained with $h=1$ as
\begin{align}\label{eq:rep}
  f(g)\ =\ \int_{\G} \textrm{d}\x\ =\ \frac{1}{\Gamma(1+n/d)}\,\int_{\R^n} \exp(-g(\x))\,\textrm{d}\x,\quad g\in \mathcal{C}_{d,n}.
\end{align}
Moreover, $f_h$ is $C^\infty$-differentiable and its derivatives are expressed as follows.

\begin{prop}\label{prop:derivative}
Let $h\in \overline{\mathcal{C}_{e,n}}$. Then for any $k\in \N$ and any $v_1,\dots, v_k\in \H_{d,n}$ we have for $g\in \mathcal{C}_{d,n}$
  \begin{equation}\label{eq:derivative}
  \begin{aligned}
    (-1)^k D_{v_1}\dots D_{v_k} f_h(g)\, &=\, \frac{\Gamma\left(1+k+\frac{n+e}{d}\right)}{\Gamma\left(1+\frac{n+e}{d}\right)}\int_{\G}h(\x) v_1(\x)\cdots v_k(\x)\, \textrm{\normalfont{d}}\x\\
    &=\, \frac{1}{\Gamma\left(1+\frac{n+e}{d}\right)}\int_{\R^n} h(\x) v_1(\x)\cdots v_k(\x) \exp(-g(\x))\,\textrm{\normalfont{d}}\x.
  \end{aligned}
\end{equation}
\end{prop}

\begin{proof}
  Since $g(\x)$ is positive for $\x\neq \boldsymbol{0}$, its minimum $g_{\min} = \min_{\Vert \x\Vert=1} g(\x)>0$ over the unit sphere in $\R^n$ is positive.
  By homogeneity, $g(\x)\geq g_{\min}\Vert \x\Vert^d$, $\x\in \R^n$.
  Therefore, derivatives of the integrand in \eqref{eq:repgen} read\smallskip
  \begin{equation*}\label{eq:der_integrand}
    \begin{aligned}
      (-1)^k D_{v_1}\dots D_{v_k} \left(h(\x) \exp(-g(\x))\right)\, =\, h(\x) v_1(\x)\cdots v_k(\x) \exp(-g(\x)),\quad \x\in \R^n.
    \end{aligned}
  \end{equation*}
  Since this is majorized by the function $h(\x)\vert v_1(\x)\cdots v_k(\x)\vert \exp(-g_{\min}\Vert \x\Vert^d)$, which is clearly integrable, dominated converge theorem implies that\smallskip
  \begin{equation}\label{eq:derivative1}
    \begin{aligned}
      (-1)^k D_{v_1}\dots D_{v_k} f_h(g)\, =\, \frac{1}{\Gamma\left(1+\frac{n+e}{d}\right)} \int_{\R^n} h(\x) v_1(\x)\cdots v_k(\x) \exp(-g(\x))\,\textrm{\normalfont{d}}\x.
    \end{aligned}
  \end{equation}
For any multiindex $\alpha\in \N^n$ formula \cite[(20)]{nonGaussian} yields
  \begin{equation*}
    \begin{aligned}
 \int_{\R^n} h(\x)\,\x^{\alpha} \exp(-g(\x))\,\textrm{\normalfont{d}}\x\ =\     \Gamma(1+(n+e+\vert \alpha\vert)/d) \int_{\G} h(\x)\, \x^{\alpha}\,\textrm{\normalfont{d}}\x.
    \end{aligned}
  \end{equation*}
From multilinearity of \eqref{eq:derivative1} in $v_1,\dots, v_k$ the remaining equality in \eqref{eq:derivative} follows.
\end{proof}

\begin{rem}
Proposition \ref{prop:derivative} implies that  $f_h: \mathcal{C}_{d,n}\rightarrow \R$ and hence the volume function \eqref{eq:volume_function} are completely monotone. Indeed, for any $k\in \N$ and positive definite forms $v_1,\dots, v_k\in \mathcal{C}_{d,n}$ the derivative $(-1)^k D_{v_1}\dots D_{v_k}f_h(g)$ is non-negative for all $g\in \mathcal{C}_{d,n}$ by \eqref{eq:derivative}.  
In fact, it is even strictly positive as one can see directly. 
\end{rem}

Next, we give an alternative proof of complete monotonicity of the volume function, that also delivers an integral representation predicted by Theorem \ref{thm:BHWC}.

\begin{proof}[Proof of Theorem \ref{thm:CM}]
As already mentioned, complete monotonicity of the volume function \eqref{eq:volume_function} is guaranteed by Proposition \ref{prop:derivative}.

Denote by $\mu$ the push-forward measure of the Lebesgue measure $\lambda$ on $\R^n$ under the Veronese map \eqref{eq:Veronese}, that is, $\mu(B)\ =\ \lambda(\boldsymbol\Theta_{d,n}^{-1}(B))$ for any Borel measurable subset $B\subseteq \R^n$. 
Then the definition \eqref{eq:Veronese} of $\boldsymbol\Theta_{d,n}$ and \cite[Thm. 3.6.1]{Bogachev} give
 \[\int_{\R^n}\exp(-g(\boldsymbol\ell))\,\textrm{d}\boldsymbol\ell\ =\ \int_{\R^n}\exp(-\langle \theta_{\boldsymbol\ell}, g\rangle) \,\textrm{d}\lambda(\boldsymbol\ell)\ =\ \int_{\mathcal{C}_{d,n}^{\,*}}\exp(-\langle \theta, g\rangle)\,\textrm{d}\mu(\theta),\]
which combined with \eqref{eq:int}  yields \eqref{eq:int_rep}.
\end{proof}

\begin{rem}
Analogously, when $h\in \overline{\mathcal{C}_{d,n}}$ has degree $d$, one can show that
\[ 
f_h(g)\ =\ \int_{\mathcal{C}_{d,n}^{\,*}} \exp\left(-\langle \theta,g\rangle\right)\frac{\langle \theta, h\rangle \,\textrm{\normalfont{d}}\mu(\theta)}{\Gamma(1+(n+d)/d)},
\]
where $\mu$ is the push-forward measure of the Lebesgue measure on $\R^n$ under the Veronese map \eqref{eq:Veronese}. 
Thus, the unique Borel measure from Theorem \ref{thm:BHWC}, that is associated to the completely monotone function $f_h$ is proportional to $\langle \cdot,h\rangle \mu$. 
\end{rem}



By the Leibniz rule for derivatives, the product of two completely monotone functions is completely monotone.
In particular, for any integer $s\in \N$ the $s$-th power $f^s$ of the volume function \eqref{eq:volume_function} is completely monotone.
Interestingly, in the setting of quadratic forms $(d=2)$, $f^s$ is completely monotone for all sufficiently large (not necessarily integer) powers $s>>1$.

\begin{prop}
The function $g\mapsto f(g)^s$ on $\mathcal{C}_{2,n}$  is completely monotone if and only if $s=0,1,2,\dots,n-2$ or $s\geq n-1$.  
\end{prop}

\begin{proof}
By Proposition \ref{prop:2d=2} below, we have that
\[f(g)^s\ =\ \left(\frac{\sqrt{\pi}^n}{\Gamma(1+n/2)}\right)^s \det(G)^{-s/2},\quad g(\x)=\x^{\mathsf T}G \x,\]
where $G$ is the positive definite matrix representing the quadratic form $g\in \mathcal{C}_{2,n}$.
The claim now directly follows from \cite[Theorem 1.3]{ScottSokal}.
\end{proof}

By its construction, the representing measure $\mu$ from Theorem \ref{thm:CM} is supported on the set $\boldsymbol\Theta_{d,n}(\R^n)\subset \H_{d,n}$ of $d$-th powers of linear forms, which is also known as  (the real part of) \emph{the Veronese cone}. 
By \cite[Prop. 5.7]{KMS}, the representing measure of the power $f^s$ of \eqref{eq:volume_function} is proportional to \emph{the $s$-th convolution power $\mu^{*s}$ of $\mu$}, 
\begin{align*}
  f^s\ =\ \prod_{j=1}^s \int_{\mathcal{C}^{\,*}_{d,n}} \exp(-\langle \theta,g\rangle)\,\frac{\textrm{d}\mu(\theta)}{\Gamma\left(1+n/d\right)}\ =\ \int_{\mathcal{C}^{\,*}_{d,n}} \exp(-\langle \theta,g\rangle)\, \frac{\textrm{d}\mu^{*s}(\theta)}{\Gamma\left(1+n/d\right)^s}.
\end{align*}
The support of the measure $\mu^{*s}$ is \emph{the $s$-th Minkowski power of $\boldsymbol\Theta_{d,n}(\R^n)$},
\begin{align}\label{eq:Minkowski}
  \supp(\mu^{*s})\,=\, \{\theta_{\boldsymbol\ell^1}+\dots+\theta_{\boldsymbol\ell^s}\,:\, \boldsymbol\ell^1,\dots, \boldsymbol\ell^s\in \R^n\}.
\end{align}

By Richter's theorem, for $s={n-1+d \choose d}$ the support \eqref{eq:Minkowski} of $\mu^{*s}$ fills the cone $\mathcal{C}_{d,n}^{\,*}$, see \cite[Satz 4]{Richter1957ParameterfreieAU} and \cite[Thm. 19]{DIDIO2022126066}. 
Actually, by \cite[Thm. 57]{DIDIO2022126066}, this already happens for some $s\leq {n-1+d\choose d}-n+1$. 

\subsection{Volume function of sums of squares}\label{sub:SOS}
Recall from \eqref{eq:SOS} that a non-negative form $g\in \overline{\mathcal{C}_{d,n}}$ is \textrm{SOS} if and only if it can be written as $g(\x)=\m_{d/2}(\x)^{\mathsf T}G\m_{d/2}(\x)$, $\x\in \R^n$, for some positive semi-definite (Gram) matrix $G$. Moreover, if $g\in \mathcal{C}_{d,n}$ is positive definite, it has a positive definite Gram matrix $G\in \textrm{PD}_M$.
Given a  non-negative form $h\in \overline{\mathcal{C}_{e,n}}$, we consider a composed map
\begin{equation}\label{eq:hsos}
\begin{aligned}
  f_h^{\,\mathrm{sos}}: \mathrm{PD}_M\ &\rightarrow\ (0,\infty),\\
  G\ &\mapsto\ f_h(\m_{d/2}(\x)^{\mathsf T}\,G\,\m_{d/2}(\x)),
\end{aligned}
\end{equation}
where $f_h$ is defined in \eqref{eq:CMgen}.
When $h=1$, we call
\begin{equation}\label{eq:sos}
\begin{aligned}
  f^{\,\mathrm{sos}}(G)\ =\ f_1^{\,\mathrm{sos}}(G)\ =\ f(\m_{d/2}(\x)^{\mathsf T}\,G\,\m_{d/2}(\x)),\quad G\in \mathrm{PD}_M,
\end{aligned}
\end{equation}
\emph{the \emph{SOS} volume function}.
The map \eqref{eq:SOS} is linear in $G$ and therefore, by Proposition \ref{prop:derivative}, \eqref{eq:hsos} (and hence also \eqref{eq:sos}) must be completely monotone on its domain $\textrm{PD}_M$.

\begin{prop}
  The function $f_h^{\,\mathrm{sos}}:\textrm{\normalfont{PD}}_M\rightarrow (0,\infty) $ is completely monotone.
\end{prop}

\begin{proof}
The claim follows directly from Proposition \ref{prop:derivative}.
  Let $V_1,\dots, V_k\in \mathrm{PD}_M$ be any positive definite matrices. Then by the definition of a directional derivative,
  \begin{align}
    (-1)^kD_{V_1}\dots D_{V_k}f_h^{\,\mathrm{sos}}(G)\ =\ (-1)^k D_{v_1}\dots D_{v_k} f_h(g)\ \geq 0,
  \end{align}
  where $g$ is given by \eqref{eq:SOS} and $v_j(\x)=\m_{d/2}(\x)^{\mathsf T}\,V_j\,\m_{d/2}(\x)\in \mathcal{C}_{d,n}$ is the positive definite form with Gram matrix $V_j$, $j=1,\dots, k$.
\end{proof}

Next, we describe an integral representation of \eqref{eq:sos}, suggested by Theorem \ref{thm:BHWC}.
For this we consider the evaluation of a sum of squares \eqref{eq:SOS} at $\boldsymbol\ell\in \R^n$, 
\begin{align}\label{eq:Trace}
    g(\boldsymbol\ell)\ =\ \m_{d/2}(\boldsymbol\ell)^{\mathsf T}\, G\,\m_{d/2}(\boldsymbol\ell)\ =\ \textrm{Tr}\left(\m_{d/2}(\boldsymbol\ell) \m_{d/2}(\boldsymbol\ell)^{\mathsf T}\, G\right),
    \end{align}
 where $\Theta_{\boldsymbol\ell}=\m_{d/2}(\boldsymbol\ell)\m_{d/2}(\boldsymbol\ell)^{\mathsf T}$ is a $M\times M$ rank-one positive semi-definite matrix.
 This motivates us to study ``the matrix analogue" of the Veronese map \eqref{eq:Veronese},
 \begin{equation}\label{eq:Vsos}
 \begin{aligned}
     \boldsymbol\Theta_{d,n}^{\,\textrm{sos}}: \R^n\ &\rightarrow\ \textrm{PD}_M^{\,*},\\
     \boldsymbol\ell\ &\mapsto\ \m_{d/2}(\boldsymbol\ell)\m_{d/2}(\boldsymbol\ell)^{\mathsf T},
 \end{aligned}
 \end{equation}
 where $\textrm{PSD}_M=\textrm{PD}_M^{\,*}=\{\Theta\in \textrm{\normalfont{Mat}}(M,\R): \textrm{Tr}(\Theta\, G)\geq 0\ \forall\ G\in \textrm{PD}_M\}$ is the dual cone to $\textrm{PD}_M$, which coincides with the closure of $\textrm{PD}_M\subset \textrm{Mat}(M,\R)$ and consists of positive semi-definite matrices. 
 Here we identify the dual space $\textrm{Mat}(M,\R)^*$ and $\mathrm{Mat}(M,\R)$ with the help of \emph{the trace inner product} $\Theta, G\in \textrm{Mat}(M,\R)\mapsto \textrm{Tr}(\Theta \,G)$.
 The main result of this subsection is an analogue of \eqref{eq:int_rep} for SOS forms.
 
 \begin{thm}
The \emph{SOS} volume function \eqref{eq:sos} admits an integral representation
 \begin{align}\label{eq:SOSint_rep}
     f^{\,\textrm{\normalfont{sos}}}(G)\ =\ \int_{\textrm{\normalfont{PSD}}_{M} } \exp(-\textrm{\normalfont{Tr}}(\Theta\,G))\,\frac{\textrm{\normalfont{d}}\mu^{\textrm{\normalfont{sos}}}(\Theta)}{\Gamma\left(1+n/d\right)},
 \end{align}
 where $\mu^{\textrm{\normalfont{sos}}}$ is the push-forward measure of the Lebesgue measure on $\R^n$ under \eqref{eq:Vsos}.
 \end{thm}

\begin{proof}
Denote by $\mu^{\textrm{sos}}$ the push-forward measure of the Lebesgue measure $\lambda$ on $\R^n$ under the map \eqref{eq:Vsos}. 
Then formula \eqref{eq:Trace} and \cite[Thm. 3.6.1]{Bogachev} give
\begin{equation}
 \begin{aligned}\int_{\R^n}\exp(-\m_{d/2}(\boldsymbol\ell)^{\mathsf T}\, G\,\m_{d/2}(\boldsymbol\ell))\,\textrm{d}\boldsymbol\ell\ &=\ \int_{\R^n}\exp(-\textrm{Tr}\left(\Theta_{\boldsymbol\ell}\, G\right)) \,\textrm{d}\lambda(\boldsymbol\ell)\\ 
 &=\ \int_{\textrm{PSD}_{M}}\exp(-\textrm{Tr}\left( \Theta\, G\right))\,\textrm{d}\mu^{\textrm{sos}}(\Theta),
 \end{aligned}
 \end{equation}
which combined with \eqref{eq:int} and \eqref{eq:sos} yields \eqref{eq:SOSint_rep}.
\end{proof}

\begin{rem}
Analogously, when $h(\x)=\m_{d/2}(\x)^{\mathsf T} H\m_{d/2}(\x)\in \Sigma_{d,n}$ is a sum of squares of degree $d$, one can show that
\begin{align}\label{eq:h_int} 
  f^{\,\mathrm{sos}}_h(g)\ =\ \int_{\textrm{\normalfont{PSD}}_M} \exp\left(-\textrm{\normalfont{Tr}}\left( \Theta\, G\right)\right)\frac{\textrm{\normalfont{Tr}}\left( \Theta\, H\right) \,\textrm{\normalfont{d}}\mu^{\textrm{\normalfont{sos}}}(\Theta)}{\Gamma(1+(n+d)/d)},
\end{align}
where $\mu^{\textrm{\normalfont{sos}}}$ is the push-forward measure of the Lebesgue measure on $\R^n$ under \eqref{eq:Vsos}. 
Thus, the unique Borel measure from Theorem \ref{thm:BHWC}, that is associated to the completely monotone function $f^{\,\mathrm{sos}}_h$ is proportional to $\textrm{\normalfont{Tr}}\left(\,\cdot\, H \right) \mu^{\textrm{\normalfont{sos}}}$. 
  \end{rem}

The case of quadratic ($d=2$) forms deserves a special attention.
In this case any $g\in \mathcal{C}_{d,n}$ can be written as $g(\x)=\x^{\mathsf T}G\x$ for a unique positive definite matrix $G$ (cf. \eqref{eq:SOS}).
By \eqref{eq:int} and elementary Gaussian integration, one obtains a closed form expression for the volume function of $g$  (see also \cite[p. 38]{Mor}).

\begin{prop}\label{prop:2d=2}
For any $n\geq 2$ we have
\begin{align*}
 f^{\,\mathrm{sos}}(G)\,=\, f(g)\,=\, \frac{\sqrt{\pi}^n}{\Gamma(1+n/2)}\frac{1}{\sqrt{\det(G)}}.
\end{align*}
\end{prop}


\begin{rem}
  Complete monotonicity of the function $G\mapsto (\det(G))^{-1/2}$ on the cone of positive definite matrices is a quite known fact, see \cite[pp. 355-356]{ScottSokal}. 
\end{rem}

We conclude this chapter with an explicit formula for \eqref{eq:hsos} in the case $d=2$. 

\begin{thm}
For any $n\geq 2$ and $h(\x)=\x^{\mathsf T} H\x$, $H\in \textrm{\normalfont{PSD}}_n$, we have
\begin{align}\label{eq:hsos_explicit}
  f_h^{\,\mathrm{sos}}(G)\ =\ \frac{\sqrt{\pi}^n}{2\,\Gamma\left(1+(n+2)/2\right)}\frac{\textrm{\normalfont{Tr}}(G^{-1}H)}{\sqrt{\det(G)}},\quad G\in \textrm{\normalfont{PD}}_n.
\end{align}
\end{thm}

\begin{proof}
  Spectral theorem applied to real symmetric matrices $G$ and $H$ yields decompositions $G=\sum_{i=1}^n \lambda_i\c_i \c_i^{\mathsf T}$ and $H=\sum_{j=1}^n \eta_j\,\d_j\d_j^{\mathsf T}$.  Here $\{\c_1,\dots, \c_d\} \subset \R^n$ and $\{\d_1,\dots, \d_n\}\subset \R^n$ are orthonormal bases consisting of eigenvectors of $G$ and, respectively, $H$, whose associated eigenvalues are $\lambda_1,\dots, \lambda_n>0$ and, respectively, $\eta_1,\dots, \eta_n\geq 0$. 
  Performing an orthogonal change of variables $\x=\sum_{k=1}^n y_k\c_k$ in the last integral of \eqref{eq:repgen}, we write
{\small\begin{equation*}\begin{aligned}
\int_{\R^n} \x^{\mathsf T} H\x \exp(-\x^{\mathsf T} G\x)\,\textrm{d}\x\ &=\ \int_{\R^n} \sum^n_{j=1} \eta_j \left(\sum_{k=1}^ny_k\c^{\mathsf T}_k\d_j\right)^2\exp\left(-\sum_{i=1}^n \lambda_i y_i^2\right)\,\textrm{d}\y \\
&=\ \int_{\R^n} \sum_{j=1}^n \eta_j \left(\sum_{k,\ell=1}^ny_ky_\ell\c^{\mathsf T}_k\d_j \c_\ell^{\mathsf T} \d_j\right)\exp\left(-\sum_{i=1}^n \lambda_i y_i^2\right)\,\textrm{d}\y\\
&=\ \int_{\R^n} \sum_{j=1}^n \eta_j\sum_{k=1}^n y_k^2\left(\c_k^{\mathsf T}\d_j\right)^2\exp\left(-\sum_{i=1}^n \lambda_i y_i^2\right)\,\textrm{d}\y,
  \end{aligned}
\end{equation*}}
where in the last step we use the fact that the integral $\int_{\R^n} y_ky_\ell \exp(-\sum_{i=1}^n \lambda_iy_i^2)\,\textrm{d}\y$ equals zero whenever $k\neq \ell$.
After one more change of variables $y_i=z_i/\sqrt{\lambda_i}$, $i=1,\dots, n$, and elementary Gaussian integration, the above integral reads
\begin{equation*}
  \begin{aligned}
   \sum_{i,j=1}^n \frac{\eta_j}{\lambda_i}
    \left(\c_i^{\mathsf T}\d_j\right)^2 \int_{\R^n} z_i^2 \exp\left(-\z^{\mathsf T}\z\right)\,\frac{\textrm{d}\z}{\sqrt{\lambda_1\cdots\lambda_n}}\ =\ \frac{\sqrt{\pi}^n}{2\sqrt{\det(G)}} \sum_{i,j=1}^n\frac{\eta_j}{\lambda_i}  \left(\c_i^{\mathsf T}\d_j\right)^2.
  \end{aligned}
\end{equation*}
The inverse matrix of $G\in\textrm{PD}_n$ is given by $G^{-1}=\sum_{i=1}^n\lambda^{-1}_i\c_i\c_i^{\mathsf T}$.
Using properties of the trace, we can rewrite the double sum in the last expression above as
\begin{align*} \sum_{i,j=1}^n\frac{\eta_j}{\lambda_i}  \left(\c_i^{\mathsf T}\d_j\right)^2\ &=\  \sum_{i,j=1}^n\frac{\eta_j}{\lambda_i}\textrm{Tr}\left((\c_i^{\mathsf T}\d_j)(\d_j^{\mathsf T}\c_i)\right)\
                                                                                            =\ \sum_{i,j=1}^n\lambda^{-1}_i\eta_j\textrm{Tr}\left(\c_i\c_i^{\mathsf T}\d_j\d_j^{\mathsf T}\right)\\
  &=\ \textrm{Tr}\left(\left(\sum_{i=1}^n\lambda^{-1}_i\c_i\c_i^{\mathsf T}\right)\left(\sum_{j=1}^n\eta_j\,\d_j\d_j^{\mathsf T}\right) \right)\ =\ \textrm{Tr}\left(G^{-1}H\right).
  \end{align*}
By gathering all together and combining with \eqref{eq:repgen} we obtain the claim \eqref{eq:hsos_explicit}.
\end{proof}

\section{Moment matrices of central Gaussian vectors and sums of squares}\label{sec:moments}


Our goal in this section is to prove Theorem \ref{thm:moments} and then derive Corollary \ref{cor:Nesterov}.
\begin{proof}[Proof of Theorem \ref{thm:moments}]
Let us consider the vector $\hat\m_{d/2}(\x)=\left(\sqrt{{d/2\choose \alpha}}\,\x^\alpha\right)_{\vert\alpha\vert={d/2}}$ of normalized monomials of degree $d/2$. 
The advantage of working with the normalized monomials consists in the fact that they (unlike $\x^\alpha$, $\vert\alpha\vert=d/2$) form an orthonormal basis of $\H_{d/2,n}$ with respect to the Bombieri inner product \eqref{eq:Bomb}. It is a well-known fact (see, e.g., \cite[p. 17]{KT2022}) that \eqref{eq:Bomb} is invariant under the standard action of $O(n)$ on $\H_{d/2,n}$ by changes of variables.
Thus, an orthogonal change of variables $\x\mapsto a\x$, $a\in O(n)$, induces a change of basis in $\H_{d/2,n}$ given by some ``big" orthogonal matrix $A\in O(M)$, $M={d/2+n-1 \choose n-1}$, that is,
\begin{align}\label{eq:transformation}
    \hat\m_{d/2}(a\x)\ =\ A \hat\m_{d/2}(\x).
    \end{align}
    The entries of the ``big" matrix $A\in O(M)$ are some forms of degree $d/2$ in the entries of the ``small" matrix $a\in O(n)$.
The matrix 
{\small\[\hat\M_d[Q]\ =\ \int_{\R^n} \hat\m_{d/2}(\y)\hat\m_{d/2}(\y)^{\mathsf T} p^{(d)}_{\,Q}(\y)\,\textrm{d}\y \ =\ \left(\int_{\R^n} \sqrt{{d/2\choose\alpha}}\sqrt{{d/2\choose \beta}} \y^{\alpha+\beta} p^{(d)}_{\,Q}(\y)\,\textrm{d}\y\right)\]}
of normalized moments of degree $d$ of $p_{\,Q}^{(d)}$ is related to $\M_d[Q]$ via the formula
\begin{align}\label{eq:relation} \hat\M_d[Q]\ =\ S \,\M_d[Q] \,S,
\end{align}
where 
\[ S\ =\ 
\begin{pmatrix}
\ddots & \cdots & 0 \\
\vdots & \sqrt{{d/2\choose \alpha}} & \vdots\\
0 & \cdots & \ddots
\end{pmatrix}\]
is the diagonal matrix of square roots of multinomial coefficients.
We now observe that  forms $\m_{d/2}(\x)^{\mathsf T} \M_d[Q]^{-1}\m_{d/2}(\x)$ and $\hat\m_{d/2}(\x)^{\mathsf T} \hat\M_d[Q]^{-1} \hat\m_{d/2}(\x)$ coincide.
Indeed, the relation \eqref{eq:relation} together with $\hat\m_{d/2}(\x)=S \m_{d/2}(\x)$ imply  
\begin{equation}\label{eq:usual=normalized}
\begin{aligned}
\m_{d/2}(\x)^{\mathsf T} \M_d[Q]^{-1}\m_{d/2}(\x)\ &=\ \hat\m_{d/2}(\x) S^{-1} \M_d[Q]^{-1} S^{-1}\hat\m_{d/2}(\x)\\ &=\ \hat\m_{d/2}(\x) \hat\M_d[Q]^{-1}\hat\m_{d/2}(\x).
\end{aligned}
\end{equation}
Next, we show that it is enough to prove the claim \eqref{eq:claim} in the case when $Q$ is diagonal. 
For this, let us choose an orthogonal $a\in O(n)$ so that the conjugate to $Q$ matrix $a^{\mathsf T} Q a=\Lambda$ is the  diagonal matrix of its eigenvalues $\lambda_1,\dots, \lambda_n>0$.
Using formula \eqref{eq:transformation}, we derive
\medskip
{\small\begin{equation}\label{eq:long_derivation}
    \begin{aligned}
    &\hat\m_{d/2}(\x)^{\mathsf T} \hat\M_d^{-1}[Q]\hat\m_{d/2}(\x)\\
    &=\ \hat\m_{d/2}(\x)^{\mathsf T}\left(\int_{\R^n}\hat\m_{d/2}(\y)\hat\m_{d/2}(\y)^{\mathsf T} \frac{\sqrt{\det(Q)}}{\sigma_d^n\sqrt{(2\pi)^n}}\exp\left(-\frac{\y^{\mathsf T} a\Lambda a^{\mathsf T}\y}{2\sigma_d^2}\right)\,\textrm{d}\y\right)^{-1}\hat\m_{d/2}(\x)\\
    &=\ \hat\m_{d/2}(\x)^{\mathsf T}\left(\int_{\R^n}\hat\m_{d/2}(a\z)\hat\m_{d/2}(a\z)^{\mathsf T} \frac{\sqrt{\det(\Lambda)}}{\sigma_d^n\sqrt{(2\pi)^n}}\exp\left(-\frac{\z^{\mathsf T} \Lambda \z}{2\sigma_d^2}\right)\,\textrm{d}\z\right)^{-1}\hat\m_{d/2}(\x)\\
    &=\ \hat\m_{d/2}(\x)^{\mathsf T}\left(\int_{\R^n}A\hat\m_{d/2}(\z)\hat\m_{d/2}(\z)^{\mathsf T} A^{\mathsf T} \frac{\sqrt{\det(\Lambda)}}{\sigma_d^n\sqrt{(2\pi)^n}}\exp\left(-\frac{\z^{\mathsf T} \Lambda \z}{2\sigma_d^2}\right)\,\textrm{d}\z\right)^{-1}\hat\m_{d/2}(\x)\\
    &=\ \hat\m_{d/2}(\x)^{\mathsf T} A\, \hat\M_d[\Lambda]^{-1}A^{\mathsf T}\hat\m_{d/2}(\x)\ =\ \hat\m_{d/2}(a^{\mathsf T}\x)^{\mathsf T} \hat\M_d[\Lambda]^{-1} \hat\m_{d/2}(a^{\mathsf T}\x),
\end{aligned}\end{equation}}\noindent
where we use orthogonality of matrices $a\in O(n)$, $A\in O(N)$, linearity of the integral and the fact that $\hat\m_{d/2}(a^{\mathsf T}\x)=A^{\mathsf T}\hat\m_{d/2}(\x)$.
Therefore, \eqref{eq:claim} holds for $\Lambda$ if and only if it holds for $Q=a\Lambda a^{\mathsf T}$, since by the last derivation and \eqref{eq:usual=normalized} we have
\begin{align*}
\left(\x^{\mathsf T} Q\x\right)^d\ &=\ \left(\x^{\mathsf T}a\Lambda a^{\mathsf T}\x\right)^d\ =\ \left((a^{\mathsf T}\x)^{\mathsf T} \Lambda (a^{\mathsf T}\x)\right)^d\quad \textrm{and}\\
\m_{d/2}(\x)^{\mathsf T} \M_d[Q]^{-1}\m_{d/2}(\x)\ &=\ \hat\m_{d/2}(\x)^{\mathsf T} \hat\M_d[Q]^{-1}\hat\m_{d/2}(\x)\\ &=\ \hat\m_{d/2}(a^{\mathsf T}\x)^{\mathsf T}\hat\M_d[\Lambda]^{-1}\hat\m_{d/2}(a^{\mathsf T}\x)\\ 
&=\ \m_{d/2}(a^{\mathsf T}\x)^{\mathsf T} \M_d[\Lambda]^{-1}\m_{d/2}(a^{\mathsf T}\x).
\end{align*}

Now, we reduce proving the claim \eqref{eq:claim} for a diagonal matrix $Q=\Lambda$ to proving it for the identity matrix $Q=\mathrm{Id}$. 
Since the eigenvalues of $Q=\Lambda$ are positive we can write them as $\lambda_i=\mu_i^2$ for some $\mu_i\in \R$, $i=1,\dots, n$. 
Let us observe that, under the coordinate-wise multiplication $\x=(x_1,\dots,x_n)\mapsto \mu*\x = (\mu_1 x_1,\dots, \mu_n x_n)$ by real numbers $\mu_1,\dots, \mu_n\in \R$, the vector of monomials transforms as
\begin{align}
    \m_{d/2}(\mu*\x)\ =\ \left( \mu^\alpha \x^\alpha \right)_{\vert\alpha\vert=d}\ =\  \m_{d/2}(\mu)*\m_{d/2}(\x)\ =\ \mathcal{N}(\mu)\, \m_{d/2}(\x),
\end{align}
where $\mathcal{N}(\mu)$ is the diagonal matrix with numbers $\mu^\alpha$, $\vert \alpha\vert=d$, on its diagonal.
We relate the forms $\m_{d/2}(\x)^{\mathsf T}\M_d[Q]^{-1}\m_{d/2}(\x)$ with $Q=\Lambda$ and $Q=\mathrm{Id}$ by 
{\small \begin{align*}
    &\m_{d/2}(\x)^{\mathsf T}\M_d[\Lambda]^{-1}\m_{d/2}(\x)\\
    &=\, \m_{d/2}(\x)^{\mathsf T}\left(\int_{\R^n} \m_{d/2}(\y)\m_{d/2}(\y)^{\mathsf T} \frac{\sqrt{\det(\Lambda)}}{\sigma_d^n\sqrt{(2\pi)^n}} \exp\left(-\frac{(\mu*\y)^{\mathsf T}(\mu*\y)}{2\sigma_d^2}\right)\,\textrm{d}\y\right)^{-1}\m_{d/2}(\x)\\
    &=\, \m_{d/2}(\x)^{\mathsf T}\left(\int_{\R^n} \m_{d/2}(\mu^{-1}*\z)\m_{d/2}(\mu^{-1}*\z)^{\mathsf T}  \exp\left(-\frac{\z^{\mathsf T}\z}{2\sigma_d^2}\right)\frac{\textrm{d}\z}{\sigma_d^n\sqrt{(2\pi)^n}}\right)^{-1}\m_{d/2}(\x)\\
    &=\, \m_{d/2}(\x)^{\mathsf T}\left(\int_{\R^n} \mathcal{N}(\mu^{-1})\m_{d/2}(\z)\m_{d/2}(\z)^{\mathsf T} \mathcal{N}(\mu^{-1}) \exp\left(-\frac{\z^{\mathsf T}\z}{2\sigma_d^2}\right) \frac{\textrm{d}\z}{\sigma_d^n\sqrt{(2\pi)^n}}\right)^{-1}\hspace{-0.3cm}\m_{d/2}(\x)\\
    &=\, \m_{d/2}(\x)^{\mathsf T} \mathcal{N}(\mu)^{\mathsf T} \M_d[\mathrm{Id}]^{-1} \mathcal{N}(\mu) \m_{d/2}(\x)\, =\, \m_{d/2}(\mu*\x)^{\mathsf T} \M_d[\mathrm{Id}]^{-1} \m_{d/2}(\mu*\x),
    \end{align*}}\noindent
    where we perform a change of variables $\z=\mu*\y$, use linearity of the integral and the formula $\mathcal{N}(\mu^{-1})=\mathcal{N}(\mu)^{-1}$ with $\mu^{-1}=(\mu_1^{-1},\dots,\mu_n^{-1})$.
This together with the formula $(\x^{\mathsf T}\Lambda \x)^d=((\mu*\x)^{\mathsf T}(\mu*\x))^d$ imply that \eqref{eq:claim} holds for a diagonal matrix $Q=\Lambda$ if and only if it holds for the identity matrix $Q=\mathrm{Id}$.
 
To prove \eqref{eq:claim} for $Q=\mathrm{Id}$, we first note that the form $\m_{d/2}(\x)^{\mathsf T}\M_d[\mathrm{Id}]^{-1}\m_{d/2}(\x)$ is invariant under orthogonal changes of variables.
Indeed, this property follows from formula \eqref{eq:long_derivation} with $Q=\Lambda=\mathrm{Id}$ and \eqref{eq:usual=normalized}, since (by definition) $a\,\mathrm{Id}\,a^{\mathsf T}=\mathrm{Id}$ holds for all orthogonal matrices $a\in O(n)$. As an $O(n)$-invariant form of degree $d$ is proportional to $(\x^{\mathsf T}\x)^{d/2}$ (see, e.g., \cite[Lemma $2.1$]{KL2020}), we obtain
\[\m_{d/2}(\x)^{\mathsf T}\M_d[\mathrm{Id}]^{-1}\m_{d/2}(\x)\ =\ c\,(\x^{\mathsf T}\x)^{d/2}\] for some constant $c\in \R^n$.
We now integrate this equality  against the measure $p^{(d)}_{\,\mathrm{Id}}(\x)\,\textrm{d}\x$. Using linearity of the integral and properties of the trace and performing elementary integration, we compute
\begin{align*}
    {d/2+n-1\choose n-1}\ &=\ \textrm{Tr}\left(\M_d[\mathrm{Id}]^{-1}\M_d[\mathrm{Id}]\right)\ =\ \sum_{\vert \alpha\vert=\vert\beta\vert=d/2} \M_d[\mathrm{Id}]^{-1}_{\alpha\beta} \M_d[\mathrm{Id}]_{\alpha\beta}\\ 
    &=\ \int_{\R^n} \sum_{\vert\alpha\vert=\vert\beta\vert=d/2} \M_d[\mathrm{Id}]^{-1}_{\alpha\beta}\, \x^{\alpha+\beta}\, p^{(d)}_{\,\mathrm{Id}}(\x)\,\textrm{d}\x\\ 
    &=\ \int_{\R^n} \m_{d/2}(\x)^{\mathsf T}\M_d[\mathrm{Id}]^{-1}\m_{d/2}(\x)\,p^{(d)}_{\,\mathrm{Id}}(\x)\,\textrm{d}\x\\
    &=\ c\int_{\R^n} (\x^{\mathsf T}\x)^{d/2} \,p^{(d)}_{\,\mathrm{Id}}(\x)\,\textrm{d}\x\\ 
    &=\ c\,\sigma_d^d \frac{1}{\sqrt{(2\pi)^n}}\int_{\R^n} (\x^{\mathsf T}\x)^{d/2} \exp\left(-\frac{\x^{\mathsf T}\x}{2}\right)\,\textrm{d}\x\\
    &=\ c\,\sigma_d^d \,\frac{\vol( \mathbb{S}^{n-1})}{\sqrt{(2\pi)^n}}\int_0^\infty r^{d+n-1} \exp\left(-\frac{r^2}{2}\right)\,\textrm{d}r\\
    &=\ c\, \frac{{d/2+n-1\choose n-1}}{\sqrt{2}^d\prod_{i=0}^{d/2-1} \left(\frac{n}{2}+i\right)} \sqrt{2}^d\frac{\Gamma\left(\frac{d+n}{2}\right)}{\Gamma\left(\frac{n}{2}\right)}\ =\ c {d/2+n-1\choose n-1},
\end{align*}
where we use the formula $\vol(\mathbb{S}^{n-1})=2\sqrt{\pi}^n/\,\Gamma\left(n/2\right)$ for the volume 
of the unit sphere $\mathbb{S}^{n-1}$.
This derivation yields $c=1$ and hence completes our proof of \eqref{eq:claim}.

It remains to prove that
\begin{align}\label{eq:rep_local}
\M_d[Q]\ =\ \frac{\int_{\R^n}\m_{d/2}(\x)\m_{d/2}(\x)^{\mathsf T} \exp\left(-k(\x^{\mathsf T}Q\x)^{d/2}\right)\textrm{d}\x}{\int_{\R^n} \exp\left(-k(\y^{\mathsf T}Q\y)^{d/2}\right)\textrm{d}\y}
\end{align}
is the matrix of degree $d$ moments of the Gaussian-like density \eqref{eq:P_G}, associated to the \textrm{SOS} form $g(\x)=(\x^{\mathsf T} Q\x)^{d/2}$. Applying a change of variables $\y=\sqrt{2}\sigma_d\, \z$ in the definition \eqref{eq:M_d_def} of $\M_d[Q]$ and then using \cite[(20)]{nonGaussian}, we write 
{\small\begin{align*}
\M_d[Q]\ &=\ \frac{\sqrt{\det(Q)}}{\sigma_d^n\sqrt{(2\pi)^n}}\int_{\R^n} \m_{d/2}(\y)\m_{d/2}(\y)^{\mathsf T} \exp\left(-\frac{\y^{\mathsf T} Q \y}{2\sigma_d^2}\right)\,\textrm{d}\y \\
&=\ \frac{\sqrt{2}^{d}\sigma_d^d\sqrt{\det(Q)}}{\sqrt{\pi^n}} \int_{\R^n}\m_{d/2}(\z)\m_{d/2}(\z)^{\mathsf T} \exp\left(-\z^{\mathsf T}Q \z\right)\,\textrm{d}\z\\
&=\ \frac{\sqrt{2}^{d}\sigma_d^d\sqrt{\det(Q)}}{\sqrt{\pi^n}}\Gamma\left(1+\frac{n+d}{2}\right) \int_{\left\{\z^{\mathsf T}Q\z\leq 1\right\}} \m_{d/2}(\z)\m_{d/2}(\z)^{\mathsf T}\, \textrm{d}\z\\
&=\ \frac{\sqrt{2}^{d}\sigma_d^d\sqrt{\det(Q)}}{\sqrt{\pi^n}}\Gamma\left(1+\frac{n+d}{2}\right) \int_{\left\{(\z^{\mathsf T}Q\z)^{d/2}\leq 1\right\}} \m_{d/2}(\z)\m_{d/2}(\z)^{\mathsf T}\, \textrm{d}\z\\
&=\ \frac{\sqrt{2}^{d}\sigma_d^d\sqrt{\det(Q)}}{\sqrt{\pi^n}}\frac{\Gamma\left(1+\frac{n+d}{2}\right)}{\Gamma\left(1+\frac{n+d}{d}\right)} \int_{\R^n}\m_{d/2}(\z)\m_{d/2}(\z)^{\mathsf T} \exp\left(-(\z^{\mathsf T}Q \z)^{d/2}\right)\,\textrm{d}\z\\
&=\ \frac{k^{1+\frac{n}{d}}\sqrt{2}^{d}\sigma_d^d\sqrt{\det(Q)}}{\sqrt{\pi^n}}\frac{d\,\Gamma\left(\frac{n+d}{2}\right)}{2\,\Gamma\left(\frac{n+d}{d}\right)} \int_{\R^n}\m_{d/2}(\z)\m_{d/2}(\z)^{\mathsf T} \exp\left(-k(\y^{\mathsf T}Q \y)^{d/2}\right)\,\textrm{d}\y,
\end{align*}}\noindent
where in the last step we change variables $\z=k^{1/d}\y$. Now, using the definition of $\sigma_d$ and formulas $k^{-1}=2{d/2+n-1\choose n}={d/2+n-1\choose n-1}d/n$, $\vol(\mathbb{S}^{n-1})=2\sqrt{\pi^n}/\Gamma(n/2)$, we simplify the constant in front of the last integral and write
\begin{align}\label{eq:M_d_local}
    \M_d[Q]\ &=\ \frac{k^{\frac{n}{d}} n \sqrt{\det(Q)}}{\vol(\mathbb{S}^{n-1})\Gamma\left(\frac{n+d}{d}\right)}  \int_{\R^n}\m_{d/2}(\z)\m_{d/2}(\z)^{\mathsf T} \exp(-k(\y^{\mathsf T}Q \y)^{d/2})\,\textrm{d}\y.
\end{align}
Performing a sequence of changes of variables and using basic properties of Gamma function, one can see that the mass of the measure $\exp(-k(\y^{\mathsf T}Q \y)^{d/2})\,\textrm{d}\y$ equals
\begin{align*}
    \int_{\R^n} \exp\left(-k(\y^{\mathsf T}Q \y)^{d/2}\right)\,\textrm{d}\y\ =\ \frac{\vol(\mathbb{S}^{n-1})\Gamma\left(\frac{n+d}{d}\right)}{k^{\frac{n}{d}}n\sqrt{\det(Q)}},
\end{align*}
which together with \eqref{eq:M_d_local} yields the claimed representation \eqref{eq:rep_local} for $\M_d[Q]$.\end{proof}

We conclude this section with a proof of Corollary \ref{cor:Nesterov}.

\begin{proof}[Proof of Corollary \ref{cor:Nesterov}]
Since the matrix $Q\in \textrm{PD}_n$ is positive definite, we have that $\y^{\mathsf T}Q\y>0$ for non-zero $\y\in \R^n\setminus \{\mathbf{0}\}$. We first perform spherical change of variables $\y=r\,\z$, $r\in (0,\infty)$, $\z\in \mathbb{S}^{n-1}$, in the integral \eqref{eq:M_d_def} defining $\M_d[Q]$ and then, after changing to $t:=\frac{\z^{\mathsf T}Q\z}{2\sigma_d^2} r^2$ in the inner integral, we integrate out the $t$ variable:
{\small \begin{align*}
\M_d[Q]\ &=\ \frac{\sqrt{\det(Q)}}{\sigma_d^n\sqrt{(2\pi)^n}}\int_{\R^n} \m_{d/2}(\y)\m_{d/2}(\y)^{\mathsf T} \exp\left(-\frac{\y^{\mathsf T} Q \y}{2\sigma_d^2}\right)\textrm{d}\y   \\
&=\ \frac{\sqrt{\det(Q)}}{\sigma_d^n\sqrt{(2\pi)^n}}\int_{\mathbb{S}^{n-1}}  \m_{d/2}(\z)\m_{d/2}(\z)^{\mathsf T} \int_0^\infty r^{d+n-1} \exp\left(-\frac{\z^{\mathsf T} Q \z}{2\sigma_d^2}r^2\right)\textrm{d}r\,\textrm{d}\mathbb{S}^{n-1}(\z)   \\
&=\ \sqrt{\det(Q)} \frac{\sqrt{2}^{d-2}\sigma_d^d}{\sqrt{\pi}^n} \Gamma\left(\frac{d+n}{2}\right) \int_{\mathbb{S}^{n-1}}\frac{\m_{d/2}(\z)\m_{d/2}(\z)^{\mathsf T}}{\sqrt{\z^{\mathsf T} Q\z}^{\,d+n}}\,\textrm{d}\mathbb{S}^{n-1}(\z)\\
&= \sqrt{\det(Q)} {d/2+n-1\choose n-1} \frac{\Gamma\left(\frac{n}{2}\right)}{2\sqrt{\pi}^n} \int_{\mathbb{S}^{n-1}}\frac{\m_{d/2}(\z)\m_{d/2}(\z)^{\mathsf T}}{\sqrt{\z^{\mathsf T} Q\z}^{\,d+n}}\,\textrm{d}\mathbb{S}^{n-1}(\z).
\end{align*}}
This, together with the formula $\vol(\mathbb{S}^{n-1})=2\sqrt{\pi}^n/\,\Gamma\left(n/2\right)$ for the volume of the sphere, implies that $\M_d[Q]$ is the matrix of degree $d$ moments of the measure $\nu$ on $\mathbb{S}^{n-1}$ from the statement of Corollary \ref{cor:Nesterov}. The rest follows from Theorem \ref{thm:moments}.
\end{proof}

\section{Comparison of $L^1$- and $L^2$-norms associated to a PD form}\label{sec:L2L1}

In this section we prove an extremal property of positive definite forms, that is stated in Theorem \ref{thm:L2L1}. 

\begin{proof}[Proof of Theorem \ref{thm:L2L1}]
For a positive definite form $g\in \mathcal{C}_{d,n}$ and any $h\in \H_{d,n}$ one has by formula $(56)$ in \cite[Lemma 3]{nonGaussian} that
\begin{align*}\label{eq:test}
\int_{\G} \vert h(\x)\vert  g(\x)\,\textrm{d}\x &=\ \frac{n+d}{n+2d}\int_{\G} \vert h(\x)\vert\, \textrm{d}\x\ =\ \frac{n+d}{n+2d}\Vert h\Vert_{L^1(\mu_g)}
\end{align*}
and in particular with $h=g$ one obtains $\Vert g\Vert^2_{L^2(\mu_g)} = \frac{n+d}{n+2d} \Vert g\Vert_{L^1(\mu_g)}$.

Let now $h\in \H_{d,n}$ be such that $\Vert h\Vert_{L^1(\mu_g)}=\Vert g\Vert_{L^1(\mu_g)}$. Then we derive
\begin{align*}
    0\  &\leq\  \Vert 
    \vert h\vert- g\Vert^2_{L^2(\mu_g)}\ =\ \Vert h\Vert^2_{L^2(\mu_g)} - 2\int_{\G} \vert h(\x)\vert g(\x)\, \textrm{d}\x+\Vert g\Vert^2_{L^2(\mu_g)}\\
    &=\  \Vert h\Vert^2_{L^2(\mu_g)}-2\frac{n+d}{n+2d}\Vert h\Vert_{L^1(\mu_g)}+\frac{n+d}{n+2d}\Vert g\Vert_{L^1(\mu_g)}\\
    &=\  \Vert h\Vert^2_{L^2(\mu_g)} - \frac{n+d}{n+2d}\Vert g\Vert_{L^1(\mu_g)}\ =\ \Vert h\Vert^2_{L^2(\mu_g)}-\Vert g\Vert^2_{L^2(\mu_g)},
\end{align*}
which yields the claim.
Moreover, if $\Vert h\Vert_{L^2(\mu_g)}=\Vert g\Vert_{L^2(\mu_g)}$, we necessarily have that $\Vert \vert h\vert- g\Vert_{L^2(\mu_g)}=0$ and hence either $h= g$ or $h=-g$.
\end{proof}

\section{Non-negative forms with sublevel sets of finite volume}\label{sec:volume}

In this section we prove Theorems \ref{thm:binary} and \ref{thm:generic} and start with some auxiliary discussion.
Let $g\in \overline{\mathcal{C}_{d,n}}$ be a non-negative form and consider the dehomogenized non-negative polynomial $\tilde g(\y)\ =\ g(\y,1)$, where $\y=(y_1,\dots, y_{n-1})\in \R^{n-1}$. 
We consider the following change of variables on $\{\x\in\R^n: x_n\neq 0\}$, 
\[ \x\ =\ (x_1,\dots, x_{n-1}, x_n)\ =\ r(y_1,\dots, y_{n-1}, 1),\quad (\y,r)\in \R^{n-1}\times \left(\R\setminus \{0\}\right),\]
which allows us to write $g(\x)=r^d\tilde g(\y)$.
Then the sublevel set of $g$ is expressed as \[\G\ =\ \{g(\x)\leq 1\}\ =\ \left\{(\y,r)\in \R^{n-1}\times \left(\R\setminus \{0\}\right): -\tilde g(\y)^{-1/d}\leq r\leq \tilde g(\y)^{-1/d}\right\}\]
and, as a consequence, its Lebesgue volume equals
\begin{align}\label{eq:vol_new_coord} 
f(g)\ =\ \int_{\G} \textrm{d}\x\ =\ \int_{\R^{n-1}} \int_{-\tilde g(\y)^{-1/d}}^{\tilde g(\y)^{-1/d}} \vert r\vert^{n-1} \textrm{d}r\,\textrm{d}\y\ =\ \frac{2}{n}\int_{\R^{n-1}} \frac{1}{\tilde g(\y)^{n/d}}\, \textrm{d}\y.
\end{align}
As we show in the following lemma, the question about finiteness of $f(g)$ has a ``local nature", at least for some class of non-negative forms $g\in \overline{\mathcal{C}_{d,n}}$.

\begin{lem}\label{lem:crit}
Let $\tilde g\in \R[y_1,\dots, y_{n-1}]$ be a non-negative polynomial of degree $d$ with only isolated real zeros $\y^{(1)},\dots \y^{(m)}\in \R^{n-1}$ and whose degree $d$ part is positive definite. 
Then for any $\alpha>\frac{n-1}{d}$ the integral $\int_{\R^{n-1}} \tilde g(\y)^{-\alpha}\,\textrm{\normalfont{d}}\y$ is finite if and only if integrals $\int_{U_i} \tilde g(\y)^{-\alpha}\,\textrm{\normalfont{d}}\y$ over some neighborhoods $U_i$ of $\y^{(i)}$, $i=1,\dots, m$, are finite.
\end{lem}

\begin{proof}
The only if direction is obvious, since $\int_{U} \tilde g(\y)^{-\alpha}\,\textrm{\normalfont{d}}\y\leq \int_{\R^{n-1}} \tilde g(\y)^{-\alpha}\,\textrm{\normalfont{d}}\y$ for any open subset $U\subseteq \R^{n-1}$. 

Let us now consider such open neighborhoods $U_i$ of $\y^{(i)}$, $i=1,\dots, m$, that satisfy $\int_{U_i} \tilde g(\y)^{-\alpha}\,\textrm{d}\y<\infty$ and take a ball $B_R=\left\{\y\in \R^{n-1}:\vert \y\vert^2=\sum_{i=1}^{n-1} y_i^2< R^2\right\}$ that contains $\y^{(1)},\dots, \y^{(m)}$. 
First, we show that the integral of $\tilde g^{-\alpha}$ over the complement of $B_R$ is finite for large $R>0$.
By our assumption, the degree $d$ homogeneous part $\tilde g_d$ of $\tilde g = \sum_{j=0}^d \tilde g_j$ is a positive definite form and hence $\tilde g_d(\y)\geq \varepsilon_d \vert \y\vert^d$ holds for some $\varepsilon_d>0$. 
For $j=0,1,\dots, d-1$ the degree $j$ homogeneous part $\tilde g_j$ of $\tilde g$ satisfies $\vert \tilde g_j(\y)\vert \leq \varepsilon_j\vert \y\vert^j$ for some $\varepsilon_j>0$.\footnote{The constant $\varepsilon_j$, $j=0,1,\dots, d-1$, can be chosen to be the maximum value of the restriction of $\vert \tilde g_j\vert$ to the sphere $\mathbb{S}^{n-2}\subset \R^{n-1}$ and $\varepsilon_d$ is a positive constant not greater than the minimum of $\tilde g_d$ over $\mathbb{S}^{n-2}$.} We put all these inequalities together and obtain
\[
\tilde g(\y)\ =\ \tilde g_d(\y)+\sum_{j=0}^{d-1} \tilde g_j(\y)\ \geq\ \tilde g_d(\y)-\sum_{j=0}^{d-1}\vert \tilde g_j(\y)\vert\ \geq \varepsilon_d \vert \y\vert^d-\sum_{j=0}^{d-1} \varepsilon_j \vert \y\vert^j,\quad \y\in \R^{n-1}.
\]
For sufficiently large $R>0$ and all points $\y\in \R^{n-1}$ with $\vert\y\vert\geq R$ the rightmost expression in the last formula can be estimated as
\begin{align}\label{eq:est_integrand}
\tilde g(\y)\ \geq\ \varepsilon_d \vert \y\vert^d-\sum_{j=0}^{d-1} \varepsilon_j \vert \y\vert^j\ \geq\ \varepsilon\vert \y\vert^d
\end{align}
for some $\varepsilon>0$.
Using \eqref{eq:est_integrand} and writing $\y=r\,\z\in \R^{n-1}\setminus B_R$, $\z\in \mathbb{S}^{n-2}$, $r\geq R$, in spherical coordinates, we estimate the integral,
\begin{align}\label{eq:integral_bound}
\int_{\R^{n-1}\setminus B_R} \frac{1}{\tilde g(\y)^{\alpha}}\,\textrm{\normalfont{d}}\y\ \leq \int_{\R^{n-1}\setminus B_R} \frac{1}{\varepsilon^{\alpha} \vert \y\vert^{d\alpha}}\,\textrm{d}\y\ =\  \frac{\textrm{vol}(\mathbb{S}^{n-2})}{\varepsilon^{\alpha}} \int_{R}^{\infty} \frac{1}{r^{d\alpha-n+2}}\,\textrm{d}r<\infty,
\end{align}
where convergence of the last integral follows from the condition $\alpha>(n-1)/d$. 
Since the function $\tilde g^{-\alpha}$ is bounded over the compact set $K=\overline{B_R\setminus \left(\bigcup_{i=1}^m U_i\right)}$, the integral $\int_{K} \tilde g(\y)^{-\alpha}\,\textrm{d}\y$ is finite.
This, together with \eqref{eq:integral_bound} and finiteness of integrals $\int_{U_i} \tilde g(\y)^{-\alpha}\,\textrm{d}\y$, $i=1,\dots, m$, finally yield
\[
\int_{\R^{n-1}} \frac{1}{\tilde g(\y)^{\alpha}}\,\textrm{\normalfont{d}}\y\ \leq\ \int_{K} \frac{1}{\tilde g(\y)^{\alpha}}\,\textrm{\normalfont{d}}\y +  \int_{\R^{n-1}\setminus B_R} \frac{1}{\tilde g(\y)^{\alpha}}\,\textrm{\normalfont{d}}\y +\sum_{i=1}^m \int_{U_i} \frac{1}{\tilde g(\y)^{\alpha}}\,\textrm{\normalfont{d}}\y\ <\ \infty.
\]
\end{proof}

\begin{rem}\label{rem:reduction}
For any $g\in \overline{\mathcal{C}_{d,n}}$ the form $g(\y,0)$ is the degree $d$ part of the (inhomogeneous) polynomial $\tilde g\in \R[y_1,\dots, y_{n-1}]$ defined by $\tilde g(\y)=g(\y,1)$.
Therefore, the degree $d$ part of $\tilde g$ is positive definite if and only if a non-negative form $g\in \overline{\mathcal{C}_{d,n}}$ has no zeros in  $\{(\y,0)\in \R^{n}: \y\neq \mathbf{0}\}$.
\end{rem}

For a non-negative binary form $g\in \overline{\mathcal{C}_{d,2}}$, the non-negative univariate polynomial $\tilde g\in \R[y]$ has only isolated real zeros, each of which is of even order. Moreover, $\tilde g$ is of degree $d$ if and only if $g$ does not vanish at $(1,0)$. This can be achieved by composing $g$ with a rotation of $\R^2$. Under such an operation the sublevel set $\G\subset \R^2$ of $g$ gets rotated and, in particular, its Lebesgue volume remains preserved. 
Thus, by \eqref{eq:vol_new_coord} and Lemma \ref{lem:crit}, finiteness of $f(g)=\vol (\G)$ can be decided by looking at the integrals of $\tilde g^{-n/d}$ over small neighborhoods of real zeros of $\tilde g$.

\begin{proof}[Proof of Theorem \ref{thm:binary}]
Let us consider a non-negative binary form $g\in \overline{\mathcal{C}_{d,2}}$.
By Remark \ref{rem:reduction} and the above discussion, without loss of generality we can assume that $\tilde g\in \R[y]$ satisfies assumptions of Lemma \ref{lem:crit}.

By \eqref{eq:vol_new_coord} the volume of $\G=\{g\leq 1\}\subset \R^2$ is proportional to the one-dimensional integral $\int_{\R} \tilde g(y)^{-2/d}\,\textrm{d}y$, where $\tilde g(y)=g(y,1)$ is the dehomogenization of $g$.
If $y'\in \R$ is a real zero of $\tilde g\in \R[y]$ of order $k$, it is possible to write $\tilde g(y)=(y-y')^k h(y)$, where $h\in \R[y]$ is a polynomial that does not vanish at $y'$.
Performing an affine change of variables $y= z+y'$, $z\in \R$, we are left with the integral $\int_{\R} z^{-2k/d} \tilde h(z)^{-2/d}\, \textrm{d}z$, where $\tilde h(z)=h(z+y')$.
For a small $\delta>0$, the function $\tilde h^{-2/d}$ is bounded (from below and from above by some positive constants) over the interval $(-\delta,\delta)$. 
Thus, $\int_{-\delta}^\delta z^{-2k/d} \tilde h(z)^{-2/d}\, \textrm{d}z<\infty$ if and only if  $\int_{-\delta}^\delta z^{-2k/d}\,\textrm{d}z<\infty$. The latter condition is in turn equivalent to $2k/d<1$, which means that the order $k$ of the zero $y'\in \R$ is smaller than $d/2$. Since the argument applies to an arbitrary real zero of $\tilde g$, the main claim  follows from Lemma \ref{lem:crit}.

For $d=4$ the above condition on the order of real zeros reads $k<2$. Since the order of a real zero of a non-negative binary form is always an even number, only positive definite forms $g$ in $\overline{\mathcal{C}_{4,2}}$ have finite $f(g)$ and hence $\mathcal{V}_{4,2}=\mathcal{C}_{4,2}$.
\end{proof}

If $g\in \overline{\mathcal{C}_{d,n}}$ is a generic non-negative form, the polynomial $\tilde g\in \R[y_1,\dots, y_n]$ can have only isolated real zeros.
In fact, a stronger result holds.

\begin{lem}\label{lem:generic}
The restriction of a  generic non-negative form $g\in \overline{\mathcal{C}_{d,n}}$ to the sphere $\mathbb{S}^{n-1}\subset \R^n$ has only finitely many zeros.
In particular, $\tilde g$ has finitely many zeros in $\R^{n-1}$ and, after possibly an orthogonal change of variables, $g$ has no zeros in $\{(\y,0)\in \R^n: \y\neq 0\}$.
\end{lem}

\begin{proof}
If $g\in\overline{\mathcal{C}_{d,n}}$ is a generic form and $\x\in \mathbb{S}^{n-1}$ is a zero of $g$, then for all unit vectors $\mathbf{v}\in \mathbb{S}^{n-1}$ that are orthogonal to $\x$, the Hessian satisfies 
\[\mathbf{v}^{\mathsf T}\textrm{Hess}_{\x}g\, \mathbf{v}=\sum_{i,j=1}^n \frac{\partial^2 g}{\partial x_i\partial x_j}(\x) v_i v_j\geq \varepsilon \]
for some $\varepsilon>0$. Since a real zero $\x\in\mathbb{S}^{n-1}$ of a non-negative form $g$ must also be its singular point, that is, $\frac{\partial g}{\partial x_i}(\x)=0$, $i=1,\dots, n$, we have 
\[ 
g(\x+t\mathbf{v}+o(t))\ =\ t^2\sum_{i,j=1}^n \frac{\partial^2 g}{\partial x_i\partial x_j}(\x)\, v_i v_j+ o(t^3)\ \geq\ t^2(\varepsilon+o(t))
\]
for points $\x+t\mathbf{v}+o(t)\in \mathbb{S}^{n-1}$ in a neighborhood of $\x\in\mathbb{S}^{n-1}$. 
In particular, for a sufficiently small $t>0$ and any $\mathbf{v}\in \mathbb{S}^{n-1}$ orthogonal to $\x$ we have $g(\x+t\mathbf{v}+o(t))>0$ and hence the zero $\x\in \mathbb{S}^{n-1}$ is isolated.
By compactness of the sphere, $g$ can have only finitely many isolated zeros in $\mathbb{S}^{n-1}$.

Injectivity of the mapping \[\y\in \R^{n-1}\mapsto \frac{(\y,1)}{\sqrt{\vert \y\vert^2+1}}\in \mathbb{S}^{n-1}\]
together with the homogeneity of $g$ implies that $\tilde g(\y)=g(\y,1)=0$ holds for only finitely many $\y\in\R^{n-1}$.
Finally, after possibly an orthogonal change of variables, the last coordinate of any real zero $\x\in\mathbb{S}^{n-1}$ of $g\in \overline{\mathcal{C}_{d,n}}$ is non-zero. 
\end{proof}

We are now ready to prove Theorem \ref{thm:generic}.

\begin{proof}[Proof of Theorem \ref{thm:generic}]
Let $g\in \overline{\mathcal{C}_{d,n}}$ be a generic non-negative form. This means in particular that for any real zero of $g$ of the form $\x=(\y,1)$ the kernel of the positive semi-definite $n\times n$ matrix $\textrm{Hess}_\x g=\left(\frac{\partial^2 g}{\partial x_i \partial x_j}(\x)\right)$ is one-dimensional and spanned by $\x$. 
Note that $\y\in \R^{n-1}$ is a zero of $\tilde g$ and the $(n-1)\times (n-1)$ Hessian matrix $\textrm{Hess}_\y \tilde g=\left(\frac{\partial^2 \tilde g}{\partial y_i\partial y_j}(\y)\right)$ is positive definite. 
Indeed, because $\tilde g\in \R[y_1,\dots, y_{n-1}]$ is non-negative, the matrix $\textrm{Hess}_\y \tilde g$ must be at least positive semi-definite. 
If $\u\in \R^{n-1}$ is a vector in the kernel of $\textrm{Hess}_\y \tilde g$, then $\w=(\u,0)\in \R^n$ must be in the kernel of $\textrm{Hess}_\x g$, because  
\[\w^{\mathsf T}\textrm{Hess}_\x g \,\w\ =\ \u^{\mathsf T} \textrm{Hess}_\y \tilde g\, \u\ =\ \boldsymbol{0}\]
and both Hessian matrices are positive semi-definite.
But then $\w=(\u,0)$ has to be proportional to $\x=(\y,1)$ and hence $\u=\mathbf{0}$. 
Positive definitedness of $\textrm{Hess}_{\y}\tilde g$ means in particular that 
\begin{align}\label{eq:Hess2} \v^{\mathsf T} \textrm{Hess}_{\y} \tilde g\, \v\ =\ \sum_{i,j=1}^{n-1}\frac{\partial^2 \tilde g}{\partial y_i \partial y_j}(\y) v_i v_j \ \geq \varepsilon
\end{align}
for some $\varepsilon>0$ and all unit vectors $\v\in \mathbb{S}^{n-2}\subset \R^{n-1}$.

Now, by Lemma \ref{lem:generic} and Remark \ref{rem:reduction}, we can assume that the non-negative polynomial $\tilde g\in \R[y_1,\dots, y_n]$ satisfies assumptions of Lemma \ref{lem:crit}. 
Thus, it is enough to prove that the integral $\int_{U} \tilde g(\z)^{-n/d}\,\textrm{d}\z$ over some neighborhood of $\y\in \R^{n-1}$ converges.
Take $U=B(\y,\delta)\subset \R^{n-1}$ to be the open ball of radius $\delta>0$ centered at $\y$. 
Since the real zero $\y$ of the non-negative polynomial $\tilde g$ is also its singular point, that is, $\frac{\partial \tilde g}{\partial y_i}(\y)=0$, $i=1,\dots, n-1$, the Taylor expansion of $\tilde g$ at the point of the ball $\y+t\v$, $\v\in \mathbb{S}^{n-2}$, $t\in [0,\delta)$, reads
\[\tilde g(\y+t\v)\ =\ t^2\sum_{i,j=1}^{n-1} \frac{\partial^2 \tilde g}{\partial y_i \partial y_j}(\y) v_i v_j + o(t^3)\ \geq\ t^2(\varepsilon+o(t)),\]
where we invoke \eqref{eq:Hess2}.
For a sufficiently small $\delta>0$ there exists $\varepsilon'>0$ so that $\varepsilon+o(t)>\varepsilon'$ for any $\y+t\v\in B(\y,\delta)$. 
This leads to an estimate
\begin{align*}\int_{B(\y,\delta)} \frac{1}{\tilde g(\z)^{n/d}}\, \textrm{d}\z\ &=\ \int_{\mathbb{S}^{n-2}}\int_0^\delta  \frac{t^{n-2}}{\tilde g(\y+t\v)^{n/d}}\, \textrm{d}t
\,\textrm{d}\mathbb{S}^{n-2}(\v)\\ &\leq\ \frac{\vol (\mathbb{S}^{n-2})}{\varepsilon'^{\,n/d}} \int_{0}^\delta t^{-2n/d+n-2}\,\textrm{d}t,
\end{align*}
where we performed a spherical change of variables.
Finally, the last integral is finite, since $-2n/d+n-2>-n/2+n-2 >-1$ whenever $d\geq 4$ and $n\geq 3$.
\end{proof}

\subsection{On the denseness of generic forms in the boundary of $\mathcal{C}_{d,n}$}

We conclude this section with a proposition, that justifies the term ``generic".
Note that our proof of this result exploits facts from classical algebraic geometry.
\begin{prop}\label{prop:nongeneric}
Non-generic non-negative forms in $\partial \mathcal{C}_{d,n}$ form a semialgebraic subset of $\H_{d,n}$ of codimension at least $2$. Moreover, this set is nowhere dense in $\partial \mathcal{C}_{d,n}\subset \H_{d,n}$, endowed with the topology induced from $(\H_{d,n},\langle \cdot,\cdot\rangle)$.
\end{prop}

\begin{proof}
Throughout this proof, $\H_{d,n}^{\mathbb{C}}=\H_{d,n}\otimes_{\R} \mathbb{C}$ denotes the space of complex $n$-variate forms of degree $d$.
Those forms that are singular at some point of \emph{the complex projective $(n-1)$-space $\mathbb{C}\textrm{\normalfont{P}}^{n-1}$} form an irreducible algebraic hypersurface
\begin{equation*}
    D_{d,n}^{\mathbb{C}}\ =\ \left\{g\in \H_{d,n}^{\mathbb{C}}\,:\, \frac{\partial g}{\partial x_1}(\x)=\dots=\frac{\partial g}{\partial x_n}(\x)=0\ \textrm{for some}\ \x\in \mathbb{C}\textrm{P}^{n-1}\right\},
\end{equation*}
that can be retrieved as the Zariski closure of the boundary $\partial \mathcal{C}_{d,n}\subset \H_{d,n}\subset \H_{d,n}^{\mathbb{C}}$ of the cone of non-negative forms, see \cite[Thm. 4.1]{NIE2012167}.
Let us consider the set 
\begin{equation}\label{eq:X^C}
    X_{d,n}^{\mathbb{C}}\ =\ \left\{g\in \H_{d,n}^{\mathbb{C}}\ \large\colon\ \begin{aligned}&
    \frac{\partial g}{\partial x_1}(\x)=\dots=\frac{\partial g}{\partial x_n}(\x)= 0\ \textrm{and}\\ & \textrm{rank}\left(\textrm{Hess}_{\x}(g)\right)\leq n-2\ \textrm{for some}\ \x\in \mathbb{C}\textrm{P}^{n-1}\end{aligned}\right\}
    \end{equation}
 of forms with a singular point $\x\in \mathbb{C}\textrm{P}^{n-1}$, at which the Hessian matrix $\textrm{Hess}_{\x}(g)=\left(\frac{\partial^2 g}{\partial x_i \partial x_j}(\x)\right)$ has corank at least $2$.
Note that $X_{d,n}^{\mathbb{C}}$ is obtained by projecting to the first coordinate the algebraic variety of pairs $(g,\x)\in \H_{d,n}^{\mathbb{C}}\times \mathbb{C}\textrm{P}^{n-1}$ satisfying conditions in \eqref{eq:X^C}.
Thus, the set $X_{d,n}^{\mathbb{C}}\subset\H_{d,n}^{\mathbb{C}}$ is algebraic by \cite[Thm. 1.11]{Shafarevich}.

Let $g(\x)=x_n^{d-2}\sum_{i=1}^{n-1}x_i^2+\frac{2}{d}\sum_{i=1}^{n-1} x_i^d$ be a degree $d$ form that we considered in Example \ref{ex:series}. 
Apart from the unique real zero $(0:\dots:0:1)$, the form $g$ (in general) has other singular points in $\mathbb{C}\textrm{P}^{n-1}$. 
All these points are described by the equations 
\begin{equation}\label{eq:system}
\begin{aligned}
x_1(x_1^{d-2}+x_n^{d-2})\ =\ 0,\ \dots,\ x_{n-1}(x_{n-1}^{d-2}+x_n^{d-2})\ =\ 0,\quad \sum_{i=1}^{n-1} x_i^2\ =\ 0.
\end{aligned}
\end{equation}
One verifies directly that the Hessian matrix $\mathrm{Hess}_{\x}g$ at any solution $\x\in\mathbb{C}^n\setminus\{\mathbf{0}\}$ of \eqref{eq:system} has corank one. 
This implies that $g\in D_{d,n}^{\mathbb{C}}\setminus X_{d,n}^{\mathbb{C}}$. 
Since the variety $X_{d,n}^{\mathbb{C}}$ is properly (as just shown) contained in the irreducible hypersurface $D_{d,n}^{\mathbb{C}}\subset \H_{d,n}^{\mathbb{C}}$, its codimension in $\H_{d,n}^{\mathbb{C}}$ must be at least 2.
In particular, the real part $X_{d,n}:=X_{d,n}^{\mathbb{C}}\cap \H_{d,n}$ is of (real) codimension at least 2 in $\H_{d,n}$.
By definition, a non-negative form $g\in \partial\mathcal{C}_{d,n}$ is non-generic if and only if it is in $X_{d,n}$. 
Thus, the semialgebraic subset $X_{d,n}\cap \partial\mathcal{C}_{d,n}\subset \H_{d,n}$ of non-generic forms has codimension at least 2.

To prove the second claim, let $U\subset \H_{d,n}$ be any open subset with $U\cap \partial\mathcal{C}_{d,n}\neq \varnothing$. 
By \cite[Rem. 2.4]{Sinn} and the proof of \cite[Lemma 2.5]{Sinn}, the open subset $U\cap \partial \mathcal{C}_{d,n}$ of the boundary of $\mathcal{C}_{d,n}$ has codimension $1$.
Since the set $U\cap (X_{d,n}\cap \partial \mathcal{C}_{d,n})$ is of codimension at least $2$, it is not dense in $U\cap \partial \mathcal{C}_{d,n}$ and hence $X_{d,n}\cap \partial\mathcal{C}_{d,n}$ is nowhere dense in $\partial\mathcal{C}_{d,n}$.
\end{proof}

\bibliographystyle{plain}

\end{document}